\documentclass{amsart}

\newtheorem{thm}{Theorem}[section]

\newtheorem{lem}[thm]{Lemma}
\newtheorem{prop}[thm]{Proposition}

\theoremstyle{definition}

\newtheorem{example}[thm]{Example}
\theoremstyle{remark}
\newtheorem{rem}[thm]{Remark}
\newtheorem{question}[thm]{Question}
\numberwithin{equation}{section}
\usepackage{amsmath,amssymb,graphicx,bbm,amsthm,hyperref,mathrsfs,tikz,palatino,changepage,enumerate,color,ifpdf,booktabs,multirow,tabularx}
\usepackage[all]{xy}
\usetikzlibrary{shapes.geometric,arrows}

\begin{document}

\title[]
{Extendable periodic automorphisms of closed surfaces over the 3-sphere}


\author{Chao Wang}
\address{School of Mathematical Sciences \& Shanghai Key Laboratory of PMMP, East China Normal University, Shanghai 200241, China}
\email{chao\_wang\_1987@126.com}

\author{Weibiao Wang}
\address{School of Mathematical Sciences, Peking University, Beijing 100871, China}
\email{wwb@pku.edu.cn}

\subjclass[2010]{57M60, 57S17, 57S25.}

\keywords{Extendable map, cyclic group action, periodic surface map, symmetry of 3-sphere.}

\begin{abstract}
  A periodic automorphism of a surface $\Sigma$ is said to be extendable over $S^3$ if it extends to a periodic automorphism of the pair $(S^3,\Sigma)$ for some possible embedding $\Sigma\hookrightarrow S^3$. We classify and construct all extendable automorphisms of closed surfaces, with orientation-reversing cases included. Moreover, they can all be induced by automorphisms of $S^3$ on Heegaard surfaces. As a by-product, the embeddings of surfaces into lens spaces are discussed.
\end{abstract}

\date{}
\maketitle

%
%
%

\section{Introduction}\label{sect:introduction}
  Let $\Sigma_g$ be a connected closed orientable surface of genus $g$. Denote the automorphism group of a manifold by ${\rm Aut}(\cdot)$. A torsion $f\in {\rm Aut}(\Sigma_g)$ is called \emph{extendable} over a 3-manifold $M$, if there exist an embedding $e:\Sigma_g \hookrightarrow M$ and a periodic automorphism $\phi$ of $M$, such that $e \circ f=\phi \circ e$. In other words, $f$ can be induced by a symmetry $\phi$ of $M$ on some embedded surface.
  
  Existing results on this topic are mainly discussed in smooth category and concerned with the case $M=S^3$ or $M=\mathbb{R}^3$. For concrete examples, see \cite{GWWZ}, where the authors determine the extendabilities over $S^3$ for all periodic maps on $\Sigma_2$.
  Equivalent conditions for a periodic orientation-preserving automorphism to be extendable over $\mathbb{R}^3$ were first given in \cite{R}, while the orientation-reversing ones are classified in \cite{C2}. The former work was resently generalized for $S^3$ in orientation-preserving category, see \cite{NWW}, and the authors put forward the following question. 
  
  \begin{question}
  	Take orientation-reversing automorphisms into consideration (i.e., either $f$ or $\phi$ in the definition, or both of them, can be orientation-reversing), then how to classify the periodic maps on $\Sigma_g$ that are extendable over $S^3$? Can the corresponding embedded surfaces always be chosen as Heegaard ones?
  \end{question}
  
  In this paper, we are to solve this problem and construct all periodic extendable maps. As a consequence, we give a positive answer for the latter half of the question. Also, we work in the smooth category, thus by the geometrization of finite group actions on 3-manifolds, each torsion $\phi\in{\rm Aut}(S^3)$ is conjugate to an orthogonal action on the standard 3-sphere $S^3\subset\mathbb{R}^4$. 
  
  Moreover, Funayoshi and Koda \cite{FK} proved that for $g\leq 2$, if a map $f_0\in{\rm Aut}(\Sigma_g)$ extends to an automorphism of $S^3$ with respect to some embedding $\Sigma_g\hookrightarrow S^3$, then $f_0$ can be realized as the restriction of some $\phi_0\in{\rm Aut}(S^3)$ on a Heegaard surface. It remains open whether that still holds for higher genus cases. Here $f_0,\phi_0$ are not necessary to be periodic, but according to a recent result \cite{NWW}, if $f_0$ is periodic, then $\phi_0$ can also be chosen as a torsion. Thus we obtain the partial answer for that problem.
  
  \begin{thm}\label{thm:Heegaard}
  	If an automorphism $\phi_0\in{\rm Aut}(S^3)$ induces a periodic map $f\in{\rm Aut}(\Sigma_g)$ with respect to some embedding $\Sigma_g\hookrightarrow S^3$, then $f$ can also be induced by a periodic automorphism $\phi$ of $S^3$ on a Heegaard surface.
  \end{thm}
  
  To classify periodic extendable maps, we first introduce some necessary invariants in Section \ref{sect:theorems}. With them we state our main results, Theorems \ref{thm:--}, \ref{thm:+-} and \ref{thm:-+}, which provide equivalent conditions for a periodic map on $\Sigma_g$ to be extendable over $S^3$. 
  Basic examples will be presented in Section \ref{sect:examples}, and all periodic extendable automorphisms can be constructed from them. In fact, in Section \ref{sect:realizations} we conduct modifications on the basic examples to get more extendable maps. Meanwhile, we prove that for a surface automorphism $f$ satisfying the extension conditions listed in the main theorems, $f$ must be conjugate to one of them. 
  In Section \ref{sect:necessity} we finally prove the necessity of the extension conditions case by case.
  Note that when a periodic automorphism $\phi\in{\rm Aut}(S^3,\Sigma_g)$ preserves the orientation of $S^3$, the quotient orbifold pair $(S^3/\phi,\Sigma_g/\phi)$ is related to an embedded surface in a lens space, so we deal with the topic in Section \ref{sect:surfaces-in-lens-space}. 
  
  Extendability characterizes whether a symmetry of a surface can be induced by those of 3-manifolds. The notion can be generalized in different ways and some efforts have been made to understand it. For instance, a finite subgroup $G$ of ${\rm Aut}(\Sigma_g)$ (or a $G$-action) is called \emph{extendable} over $S^3$ with respect to an embedding $e:\Sigma_g \hookrightarrow S^3$, if there exists a group monomorphism $\varphi$ from $G$ to ${\rm Aut}(S^3)$, such that $\varphi(h) \circ e=e \circ h$ holds for each $h \in G$. The maximum order of extendable groups for fixed genus $g$ is discussed in \cite{WWZZ1,WWZZ2,WZ}. And as a contrast to Theorem \ref{thm:Heegaard}, there are extendable finite group actions on $\Sigma_{21}$ and $\Sigma_{481}$ such that their extensions can not be realized on Heegaard surfaces \cite{WWZZ2}.
  
  \vspace{3mm}
  
  Acknowledgements: We thank Professor Shicheng Wang of Peking University for his suggestions and concerns on our work. The first author is supported by Science and Technology Commission of Shanghai Municipality (STCSM), grant No. 18dz2271000. The second author is supported by National Natural Science Foundation of China (NSFC) under grant No. 11871078. 

\section{Classification of periodic maps and main theorems}\label{sect:theorems}
  The classification of periodic maps on closed orientable surfaces was finished by Yokoyama \cite{Y1,Y2,Y} and Costa \cite{C1,C3} independently. Indeed, Yokoyama completed that for all compact surfaces. We introduce the involved invariants and present the result here in a convenient way for our task. And the notations will be used throughout the paper.
  
  Given a periodic map $f\in {\rm Aut}(\Sigma_g)$ of order $n$, we obtain an orbifold $\Sigma_g/f$. Suppose $\Sigma_g/f$ has $s$ isolated singular points of indices $n_1,n_2,\cdots,n_s$ respectively, and the underlying space $|\Sigma_g/f|$ is a compact surface of genus $h$, with $b$ boundary components. Then there is a canonical presentation of the orbifold fundamental group $\pi_1(\Sigma_g/f)$:
  \begin{displaymath}
  \begin{split}
  \pi_1(\Sigma_g/f)=&
  \langle \alpha_1,\beta_1,\cdots,\alpha_h,\beta_h, \eta_1,\epsilon_1,\cdots,\eta_b,\epsilon_b,\xi_1,\cdots,\xi_s |\\
  &\prod_{i=1}^{h}[\alpha_i,\beta_i]\prod_{j=1}^{b}\epsilon_j\prod_{k=1}^{s}\xi_k=1,\,
  \xi_k^{n_k}=1(1\leq k\leq s),\\ &\eta_j=\epsilon_j^{-1}\eta_j\epsilon_j,\, \eta_j^2=1(1\leq j\leq b) \rangle,
  \end{split}
  \end{displaymath} 
  if $|\Sigma_g/f|$ is orientable; or 
  \begin{displaymath}
  \begin{split}
  \pi_1(\Sigma_g/f)=&
  \langle \delta_1,\cdots,\delta_h, \eta_1,\epsilon_1,\cdots,\eta_b,\epsilon_b,\xi_1,\cdots,\xi_s |\\
  &\prod_{i=1}^{h}\delta_i^2\prod_{j=1}^{b}\epsilon_j\prod_{k=1}^{s}\xi_k=1,\,
  \xi_k^{n_k}=1(1\leq k\leq s),\\ &\eta_j=\epsilon_j^{-1}\eta_j\epsilon_j,\, \eta_j^2=1(1\leq j\leq b) \rangle,
  \end{split}
  \end{displaymath}
  if $|\Sigma_g/f|$ is non-orientable. Here $\epsilon_1,\cdots,\epsilon_b$ correspond to the boundary components, and $\xi_1,\cdots,\xi_s$ can be represented by loops surrounding the corresponding singular points. The collection of such generators,
  \begin{displaymath}
    \mathcal{G}=\{\alpha_1,\beta_1,\cdots,\alpha_h,\beta_h,  \eta_1,\epsilon_1,\cdots,\eta_b,\epsilon_b,\xi_1,\cdots,\xi_s\},
  \end{displaymath}
  or 
  \begin{displaymath}
    \mathcal{G}=\{\delta_1,\cdots,\delta_h, \eta_1,\epsilon_1,\cdots,\eta_b,\epsilon_b,\xi_1,\cdots,\xi_s\},
  \end{displaymath}
  is called a \emph{canonical generator system} of $\pi_1(\Sigma_g/f)$. 
  
  The action of the finite cyclic group $\langle f \rangle$ on $\Sigma_g$ induces a short exact sequence
  \begin{displaymath}
    1\to\pi_1(\Sigma_g)\xrightarrow{\rho_*}\pi_1(\Sigma_g/f)\xrightarrow{\psi}\langle f\rangle\to 1,
  \end{displaymath}
  where $\rho:\Sigma_g\to\Sigma_g/f$ is the quotient map.
  The following proposition shows that $\psi$ plays a key role in the classification of periodic maps up to conjugacy.
  
  \begin{prop}[Theorem 2.2 in \cite{GWWZ} and Lemma 4.1 in \cite{C1}]\label{prop:conjugacy}
  	Suppose $f,f'\in {\rm Aut}(\Sigma_g)$ are periodic maps of the same order and with epimorphisms $\psi,\psi'$ respectively. $f,f'$ are conjugate if and only if there exists an orbifold homeomorphism $H:\Sigma_g/f\to\Sigma_g/f'$ such that $\iota\circ\psi=\psi'\circ H_*$, where $\iota:\langle f\rangle\to \langle f'\rangle$ is defined by $\iota(f)=f'$.
  \end{prop}

  Given a canonical generator system $\mathcal{G}$ of $\pi_1(\Sigma_g/f)$, we obtain a collection of elements $\psi(\gamma)(\gamma\in\mathcal{G})$ in $\langle f\rangle$. 
  Mapping $f$ to the generator $1$ of the addative group $\mathbb{Z}_n$, we identify $\langle f \rangle$ with $\mathbb{Z}_n$. Without ambiguity, we use integers to represent their images under the modulo $n$ homomorphism $\mathbb{Z}\to \mathbb{Z}_n$. 
  Then the \emph{isotropy invariant} of $f$ is defined to be
  \begin{displaymath}
	  \pm(\psi(\epsilon_1),\psi(\epsilon_2),\cdots,\psi(\epsilon_b);\,\psi(\xi_1),\psi(\xi_2),\cdots,\psi(\xi_s))
  \end{displaymath}
  if $|\Sigma_g/f|$ is orientable, or
  \begin{displaymath}
	  (\pm\psi(\epsilon_1),\pm\psi(\epsilon_2),\cdots,\pm\psi(\epsilon_b);\,\pm\psi(\xi_1),\pm\psi(\xi_2),\cdots,\pm\psi(\xi_s))
  \end{displaymath}
  if $|\Sigma_g/f|$ is non-orientable.
  
  To classify periodic maps on closed surfaces, we still need two more invariants for two special cases. 
  \begin{description}
  	\item[Case 1] $n/2$ is even, $|\Sigma_g/f|$ is non-orientable (thus closed), and there is no $\pm n/2$ among $\psi(\xi_1),\cdots,\psi(\xi_s)$. Then we define $h_1(f)$ as 
  	\begin{displaymath}
  	h_1(f)=\sum_{i=1}^{h}\psi(\delta_i)+\sum_{k=1}^{s}\chi_k\psi(\xi_k),
  	\end{displaymath}
  	where 
  	\begin{displaymath}
  	\chi_k=
  	\begin{cases}
  	0, & \textrm{if }\psi(\xi_k)\in\{2,4,\cdots,\frac{n}{2}-2\}\subset\mathbb{Z}_n;\\
  	1, & \textrm{if }\psi(\xi_k)\in\{\frac{n}{2}+2,\frac{n}{2}+4,\cdots,n-2\}\subset\mathbb{Z}_n.
  	\end{cases}
  	\end{displaymath}
    Note that each $\psi(\delta_i)$ is odd and each $\psi(\xi_k)$ is even, so $h_1(f)$ has the same parity with $h$. 
    
    \item[Case 2] $|\Sigma_g/f|$ is non-orientable with genus $h=2$. We choose
    \begin{displaymath}
    m={\rm gcd}(\psi(\delta_1)+\psi(\delta_2),\psi(\epsilon_1),\cdots,\psi(\epsilon_b),\psi(\xi_1),\cdots,\psi(\xi_s),n),
    \end{displaymath}
    where ${\rm gcd}(\cdots)$ means the greatest common divisor.
    Let $f_m$ be the order $m$ map induced by $f$ on $|\Sigma_g/f^m|$, and 
    \begin{displaymath}
    \begin{split}
    \psi_m:\pi_1(|\Sigma_g/f^m|/f_m)\to\langle f_m\rangle &=\mathbb{Z}_m\\
    f_m &\leftrightarrow 1
    \end{split} 
    \end{displaymath}
    be the corresponding epimorphism.
    Then we define $h_2(f)$ as a multiset
    \begin{displaymath}
    h_2(f)=\{\pm\psi_m(\delta_1),\pm\psi_m(\delta_2)\}\subset\mathbb{Z}_m.
    \end{displaymath}
  \end{description}
  
  The following theorem tells that $f\in{\rm Aut}(\Sigma_g)$ is determined up to conjugacy by the topology of $|\Sigma_g/f|$, the isotropy invariant, and $h_1(f),h_2(f)$ if defined. Besides, their values are independent of the choise of the canonical generator system $\mathcal{G}$ for $\pi_1(\Sigma_g/f)$.
  
  \begin{thm}[\textbf{Classification theorem}]\label{thm:classification}
  	Let $f,f'\in{\rm Aut}(\Sigma_g)$ be two periodic maps of the same order $n$, and $|\Sigma_g/f|\cong|\Sigma_g/f'|$. 
  	
  	(1) Suppose $|\Sigma_g/f|,|\Sigma_g/f'|$ are orientable. $f,f'$ are conjugate if and only if they have the same isotropy invariant, up to a reordering of the singular points and boundary components of the corresponding orbifolds.
  	
  	(2) Suppose $|\Sigma_g/f|,|\Sigma_g/f'|$ are non-orientable.  $f,f'$ are conjugate if and only if they satisfy the following conditions:
  	\begin{enumerate}[\hspace{1.5em}(i)]
  		\item $f,f'$ have the same isotropy invariant, up to a reordering of the singular points and boundary components of the corresponding orbifolds;
  		\item if $n/2$ is even and there is no $\pm n/2$ in the isotropy invariant, then $h_1(f)=h_1(f')$;
  		\item if $|\Sigma_g/f|,|\Sigma_g/f'|$ have genus 2, then $h_2(f)=h_2(f')$. 
  	\end{enumerate}
  \end{thm}

  If $f\in {\rm Aut}(\Sigma_g)$ extends to $\phi\in{\rm Aut}(S^3)$ with some embedding $\Sigma_g\hookrightarrow S^3$, then there are four types of $(f,\phi)$:
  
  \begin{itemize}
  	\item type $(+,+)$: $f,\phi$ are both orientation-preserving;
  	\item type $(-,-)$: $f,\phi$ are both orientation-reversing;
  	\item type $(+,-)$: $f$ preserves the orientation of $\Sigma_g$ while $\phi$ reverses that of $S^3$; 
  	\item type $(-,+)$: $f$ reverses the orientation of $\Sigma_g$ while $\phi$ preserves that of $S^3$.
  \end{itemize}
  
  Now we can state our main theorems with the notations listed above. 
  The first one is a recent result from \cite{NWW}. We just put it here for the sake of completeness.
  
  \begin{thm}\label{thm:++}
  	A periodic orientation-preserving map $f\in {\rm Aut}(\Sigma_g)$ is extendable over $S^3$ in type $(+,+)$ if and only if the isotropy invariant $\pm(\psi(\xi_1),\psi(\xi_2),\cdots,\psi(\xi_s))$ is 
  	\begin{displaymath}
	  	\pm(\underbrace{\alpha,\alpha,\cdots,\alpha}_{t},\underbrace{-\alpha,-\alpha,\cdots,-\alpha}_{t},\underbrace{\beta,\beta,\cdots,\beta}_{s/2-t},\underbrace{-\beta,-\beta,\cdots,-\beta}_{s/2-t}) 
  	\end{displaymath}
  	up to rearrangement, where $\alpha,\beta$ are elements of coprime orders $p,q$ in $\mathbb{Z}_n$ respectively, and $0\leq 2t\leq s$.
  	
  	Moreover, in that case the conjugacy class of $\langle f\rangle$ is uniquely determined by its period $n$, the genus $h$ of the quotient orbifold $\Sigma_g/f$, and the parameters $s,t,p,q$.
  \end{thm}

  \begin{thm}\label{thm:--}
  	A periodic orientation-reversing map $f\in {\rm Aut}(\Sigma_g)$ is extendable over $S^3$ in type $(-,-)$ if and only if there is a generator $\alpha$ of $\mathbb{Z}_n$, such that there are only $\pm 2\alpha$ and $0$ in the isotropy invariant and one of the followings happens:
  	\begin{enumerate}
  		\item $n/2$ is odd and $|\Sigma_g/f|$ is orientable.
  		If $n>2$, then up to rearrangement the isotropy invariant $\pm(\psi(\epsilon_1),\psi(\epsilon_2),\cdots,\psi(\epsilon_b);\,\psi(\xi_1),\psi(\xi_2),\cdots,\psi(\xi_s))$ is 
  		\begin{displaymath}
  		\pm(\underbrace{2\alpha,\cdots,2\alpha}_{t-\lceil\frac{s}{2}\rceil},\underbrace{-2\alpha,\cdots,-2\alpha}_{t-\lfloor\frac{s}{2}\rfloor},\underbrace{0,\cdots,0}_{s+b-2t};\,\underbrace{2\alpha,\cdots,2\alpha}_{\lceil\frac{s}{2}\rceil},\underbrace{-2\alpha,\cdots,-2\alpha}_{\lfloor\frac{s}{2}\rfloor}),
  		\end{displaymath}
  		where $\lceil\frac{s}{2}\rceil\leq t\leq \lfloor\frac{s+b}{2}\rfloor$.
  		
  		\item $n/2$ is odd and $|\Sigma_g/f|$ is non-orientable without boundary.
  		If $n>2$ then $h,s$ have the same parity.
  		
  		\item $n/2$ is even, and $|\Sigma_g/f|$ is non-orientable without boundary.
  		If $n>4$ then
  		\begin{displaymath}
  		h_1(f)=
  		\begin{cases}
  		-s\alpha, & \textrm{if } 2\alpha\in\{2,4,\cdots,\frac{n}{2}-2\}\subset \mathbb{Z}_n;\\
  		s\alpha, & \textrm{if } 2\alpha\in\{\frac{n}{2}+2,\frac{n}{2}+4,\cdots,n-2\}\subset\mathbb{Z}_n.
  		\end{cases}  		
  		\end{displaymath}
  		And if $n=4,s=0$, then $h_1(f)=0\in\mathbb{Z}_4$.
  	\end{enumerate}
  
  Moreover, in that case the conjugacy class of $\langle f\rangle$ is uniquely determined by $n,h,b,s,t$ if (1) happens; and by $n,h,s$ if (2) or (3) happens.
  \end{thm}

  \begin{thm}\label{thm:+-}
  	A periodic orientation-preserving map $f\in {\rm Aut}(\Sigma_g)$ is extendable over $S^3$ in type $(+,-)$ if and only if $n$ is even, $s\geq 2$, and there is a generator $\alpha$ of $\mathbb{Z}_n$, such that
  	up to rearrangement the isotropy invariant $\pm(\psi(\xi_1),\psi(\xi_2),\cdots,\psi(\xi_s))$ is 
  	\begin{displaymath}
  	\pm(\alpha,(-1)^{s-1}\alpha,-2\alpha,2\alpha,-2\alpha,2\alpha,\cdots,(-1)^{s}\cdot 2\alpha)
  	\end{displaymath}
  	(if $n=2$, then $s=2$ and the isotropy invariant is $\pm(\alpha,-\alpha)$). 
  	
  	Moreover, in that case the conjugacy class of $\langle f\rangle$ is uniquely determined by $n,h,s$.
  \end{thm}

  \begin{thm}\label{thm:-+}
  	A periodic orientation-reversing map $f\in {\rm Aut}(\Sigma_g)$ is extendable over $S^3$ in type $(-,+)$ if and only if one of the followings happens:
  	\begin{enumerate}
  		\item $n/2$ is odd, and $|\Sigma_g/f|$ is orientable with non-empty and connected boundary (i.e., $b=1$). If $n>2$, then $s$ is odd and there exists a generator $\alpha$ of $\mathbb{Z}_n$, such that up to rearrangement the isotropy invariant $\pm(\psi(\epsilon_1);\,\psi(\xi_1),\psi(\xi_2),\cdots,\psi(\xi_s))$ is 
  		\begin{displaymath}
  		(2\alpha;\,\underbrace{2\alpha,\cdots,2\alpha}_{(s-1)/2},\underbrace{-2\alpha,-2\alpha,\cdots,-2\alpha}_{(s+1)/2}).
  		\end{displaymath}
  	
	  	\item $n/2$ is odd, and $|\Sigma_g/f|$ is non-orientable with non-empty and connected boundary (i.e., $b=1$). If $s>0$, then $s$ is odd, 
	  	and there exist a factor $l$ of $n/2$ and a generator $\alpha$ of $\mathbb{Z}_n$, such that the isotropy invariant $(\pm\psi(\epsilon_1);\,\pm\psi(\xi_1),\cdots,\pm\psi(\xi_s))$ is 
	  	\begin{displaymath}
	  	(\pm 2\alpha;\,\pm 2l\alpha,\pm 2l\alpha,\cdots,\pm 2l\alpha).
	  	\end{displaymath}
	  	
	  	\item $|\Sigma_g/f|$ is non-orientable without boundary (i.e., $b=0$), then up to rearrangement the isotropy invariant $(\pm\psi(\xi_1),\cdots,\pm\psi(\xi_s))$ is 
	  	\begin{displaymath}
	  	(\underbrace{\pm\beta,\cdots,\pm\beta}_{t},\underbrace{\pm\gamma,\cdots,\pm\gamma}_{s-t}),
	  	\end{displaymath}
	  	where $0\leq t\leq s$ and $\beta,\gamma\in\mathbb{Z}_n$ has coprime orders $p,q$ respectively. If $t=0$ then set $p=1$, otherwise $t$ should be odd; if $t=s$ then set $q=1$, otherwise $s-t$ should be odd. Let $n=pql$, then $l$ is even and $l/2$ has the same parity with $h$. If $h=1$ then $l=2$. Besides, 
	  	the following conditions hold:	
	  	\begin{enumerate}[(i)]
	  		\item If $h_1(f)$ is defined, i.e., $n/2$ is even and $2\notin\{p,q\}$, then $h_1(f)\equiv l/2\,({\rm mod}\,l)$. 
	  		
	  		\item If $h=2$, without loss of generality, assume $p$ is odd. Then $f$ is conjugate to some power of the map $f_0$ whose invariants are:
	  		\begin{displaymath}
	  		\text{isotropy invariant}=(\underbrace{\pm ql,\cdots,\pm ql}_{t},\underbrace{\pm pl,\cdots,\pm pl}_{s-t});
	  		\end{displaymath}
	  		\begin{displaymath}
	  		h_1(f_0)\text{ (if defined)}=\frac{n}{2}-\max\{t,1\}\frac{ql}{2}-\max\{s-t,1\}\frac{pl}{2}\in\mathbb{Z}_{n},
	  		\end{displaymath}
	  		and
	  		\begin{displaymath}
	  		h_2(f_0)=\{\pm k, \pm(l/2-k)\}\subset\mathbb{Z}_{l/2},
	  		\end{displaymath}
	  		where $k$ is the smallest positive integer that satisfies 
	  		\begin{displaymath}
	  		\left\{
	  		\begin{array}{l}
	  		k \equiv qm_0 \,({\rm mod}\,p),\\
	  		k \equiv p \,({\rm mod}\,2q),\\
	  		{\rm gcd}(k,n) =1,
	  		\end{array} 
	  		\right.
	  		\end{displaymath}
	  		and $m_0$ is the smallest positive integer that satisfies 
	  		\begin{displaymath}
	  		\left\{\begin{array}{l}
	  		m_0\equiv l/2+1 \,({\rm mod}\,l),\\
	  		{\rm gcd}(m_0,p)=1.
	  		\end{array}
	  		\right.
	  		\end{displaymath}
	  		
	  	\end{enumerate}
  	\end{enumerate}
  
  Moreover, the conjugacy class of such $\langle f\rangle$ is uniquely determined by the parameters $n,h,s$ if (1) happens; by $n,h,s,l$ if (2) happens; and by $n,h,s,t,p,q$ if (3) happens. 
  \end{thm}

\section{Basic examples of extendable maps}\label{sect:examples}
  Before proving the main theorems, we give some basic examples of extendable maps. In each of them we choose a $\phi\in{\rm Aut}(S^3)$ of order $n$ and an embedded surface $\Sigma$ which is $\phi$-invariant, thus $\phi$ induces a surface automorphism $\phi|_\Sigma$. 
  In next section we will show that every periodic extendable map can be constructed from these examples just with some inessential modifications. 
  
  Two models for $S^3$ will be used. One is $\mathbb{R}^3\cup\{\infty\}$ where $\phi$ acts as an element of $O(3)$. The other is the unit sphere in $\mathbb{C}^2$:
  \begin{displaymath}
  \{(w_1,w_2)\in\mathbb{C}^2:|w_1|^2+|w_2|^2=1\}.
  \end{displaymath}
  Write $w_1=x_1+iy_1,w_2=x_2+iy_2,x_1,y_1,x_2,y_2\in\mathbb{R}$, then $\phi$ acts as an element of $O(4)$ on $S^3\subset\mathbb{C}^2\cong\mathbb{R}^4\cong\mathbb{R}\langle x_1,y_1,x_2,y_2\rangle$.
  The two models are connected by a stereographic projection in $\mathbb{C}^2\cong\mathbb{R}^4$:
  \begin{displaymath}
  \begin{array}{ccc}
  \{(w_1,w_2)\in\mathbb{C}^2:|w_1|^2+|w_2|^2=1\} &\to &\mathbb{R}^3\cup\{\infty\}\\
  (x_1+iy_1,x_2+iy_2) &\mapsto &(\frac{x_1}{1+x_2},\frac{y_1}{1+x_2},\frac{y_2}{1+x_2}).\\
  \end{array}
  \end{displaymath} 
  It identifies the points $(0,-1),(0,1)\in\mathbb{C}^2$, the circles
  \begin{align*}
    &\{(w_1,w_2)\in \mathbb{C}^2:w_1=0,|w_2|=1\},\\
    &\{(w_1,w_2)\in\mathbb{C}^2:|w_1|=1,w_2=0\},
  \end{align*} 
  and the sphere 
  \begin{displaymath}
    \{(w_1,w_2)\in\mathbb{C}^2:{\rm Re}(w_2)=0,|w_1|^2+|w_2|^2=1\}
  \end{displaymath}
  in the second model to $\infty,(0,0,0)$, $z$-axis$\,\cup\{\infty\}$, the unit circle on the $xy$-plane and the unit sphere in $\mathbb{R}^3\cup\{\infty\}$. For convenience, denote them by $\infty,0,Z,S^1,S^2$ respectively. 
  
\subsection{Type $(+,+)$}\label{subsect:example++}
  \begin{example}\label{ex:++}
  	$\phi$ acts on $S^3\subset\mathbb{C}^2$ as $\phi(w_1,w_2)=(w_1e^{\frac{2\pi i}{n}},w_2e^{\frac{2\pi i}{n}})$, i.e., a $\frac{2\pi}{n}$-rotation on both coordinate components. Let $\Sigma$ be the torus
  	\begin{displaymath}
  	T=\{(w_1,w_2)\in\mathbb{C}^2:|w_1|=|w_2|=\frac{\sqrt{2}}{2}\}.
  	\end{displaymath}
  	As long as $n>1$ (as is always assumed), $\phi^m(m=1,2,\cdots,n-1)$ has no fixed point on $S^3$, thus $\Sigma/\phi$ is a closed orientable surface, and by the Riemann-Hurwitz formula (or just by topological observation) we know it is a torus.
  	So $\phi|_\Sigma$ satisfies the conditions in Theorem \ref{thm:++} with $g=1,h=1,s=t=0,p=q=1$.
  \end{example}

\subsection{Type $(-,-)$}\label{subsect:example--}
  \begin{example}\label{ex:--}
  	$S^3=\mathbb{R}^3\cup\{\infty\}$. $n$ is even. $\phi\in O(3)$ has the matrix
  	\begin{displaymath}
  	\left(
  	\begin{array}{ccc}
  	\cos\frac{2p\pi}{n} & -\sin\frac{2p\pi}{n} & \\
  	\sin\frac{2p\pi}{n} & \cos\frac{2p\pi}{n} & \\
  	& & -1  
  	\end{array}
  	\right),
  	\end{displaymath}
  	i.e., $\phi$ is the composition of a $\frac{2p\pi}{n}$-rotation around the $z$-axis and a reflection across the $xy$-plane.
  	
  	(1) Let $\Sigma$ be the unit sphere
  	 \begin{displaymath}
  	 S^2=\{(x,y,z)\in\mathbb{R}^3:x^2+y^2+z^2=1\}.
  	 \end{displaymath}
  	 
  	\begin{enumerate}[\hspace{1.5em}(i)]
  		\item If $p=1$, then $\phi|_\Sigma$ satisfies Theorem \ref{thm:--} (2) or (3) with $g=0,h=1,s=1$ and $\alpha=1$. Here we do not check $h_1(\phi|_\Sigma)$ for it must be as in the theorem, according to Section \ref{sect:necessity}. In the degenerate case $n=2$, actually we have $s=0$, which is inessential in our discussion. Similarly, we will also omit the calculation of the invariants $h_1,h_2$ and ignore the degenerate cases below.
  		\item If $p=2$ and $n/2$ is odd, then $\phi|_\Sigma$ satisfies Theorem \ref{thm:--} (1) with
  		$g=0,h=0,b=1,s=1,t=1$ and $\alpha=1$.
  	\end{enumerate}
  	
  	(2) Let $\Sigma$ be the boundary of a $\phi-$invariant regular neighborhood of the circle 
  	\begin{displaymath}
  	S^1=\{(x,y,z)\in\mathbb{R}^3:z=0,x^2+y^2=1\}.
  	\end{displaymath}
  	For instance, $\Sigma$ can be chosen as the torus $T$ in Example \ref{ex:++}. 
  	\begin{enumerate}[\hspace{1.5em}(i)]
  		\item If $p=1$, then $\phi|_\Sigma$ satisfies Theorem \ref{thm:--} (2) or (3) with
  		$g=1,h=2,s=0$ and $\alpha=1$.
  		\item If $p=2$ and $n/2$ is odd, then $\phi|_\Sigma$ satisfies Theorem \ref{thm:--} (1) with
  		$g=1,h=0,b=2,s=0,t=1$ and $\alpha=1$.
  	\end{enumerate}
  
  	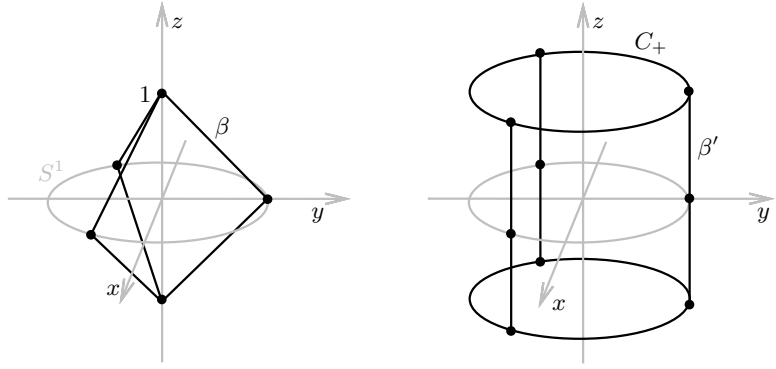
\begin{figure}[htbp]
  		\centering
  		\setlength{\unitlength}{1bp}%
  		\begin{picture}(298.60, 146.24)(0,0)
  			\put(0,0){\includegraphics{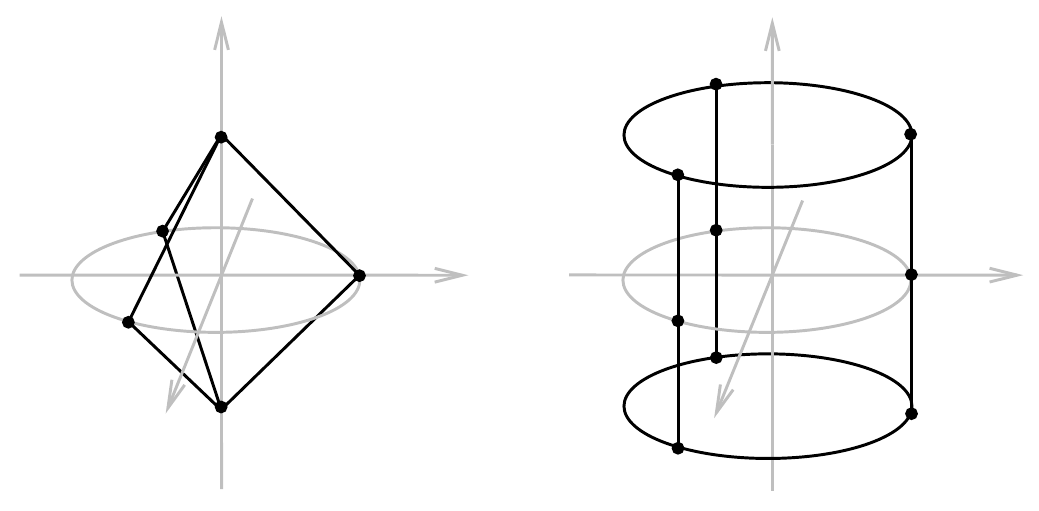}}
  			\put(55.16,104.45){\fontsize{8.83}{10.60}\selectfont $1$}
  			\put(67.18,132.02){\fontsize{8.83}{10.60}\selectfont $z$}
  			\put(42.84,31.61){\fontsize{8.83}{10.60}\selectfont $x$}
  			\put(120.13,60.40){\fontsize{8.83}{10.60}\selectfont $y$}
  			\put(83.62,92.33){\fontsize{8.83}{10.60}\selectfont $\beta$}
  			\put(225.82,132.33){\fontsize{8.83}{10.60}\selectfont $z$}
  			\put(241.83,124.17){\rotatebox{0.00}{\fontsize{8.83}{10.60}\selectfont \smash{\makebox[0pt][l]{$C_+$}}}}
  			\put(210.85,25.93){\fontsize{8.83}{10.60}\selectfont $x$}
  			\put(288.04,60.40){\fontsize{8.83}{10.60}\selectfont $y$}
  			\put(265.53,84.21){\fontsize{8.83}{10.60}\selectfont $\beta'$}
  			\put(16.95,74.48){\fontsize{8.83}{10.60}\selectfont \textcolor[rgb]{0.75294, 0.75294, 0.75294}{$S^1$}}
  		\end{picture}%
  		\caption{\label{fig:--_graphs}%
  			$\phi-$invariant graphs ($n=6$).}
  	\end{figure}
  
	(3) Suppose $p=2$ and $n/2$ is odd. Let $\beta$ be the line segment connecting the points $(0,0,1),(0,1,0)\in\mathbb{R}^3$, then the union of $\phi^m(\beta)(1\leq m\leq n)$ is a connected graph (Figure \ref{fig:--_graphs}, left). Choose a $\phi-$invariant regular neighborhood and let $\Sigma$  be its boundary, then $\phi|_\Sigma$ satisfies Theorem \ref{thm:--} (1) with $g=n/2-1,h=0,b=1,s=2,t=1$ and $\alpha=1$.
	
	(4) Suppose $p=2$ and $n/2$ is odd. Let $C_+$ be the circle 
	\begin{displaymath}
	\{(x,y,z)\in\mathbb{R}^3:z=1,x^2+y^2=1\}
	\end{displaymath}
	and $\beta'$ be the line segment connecting the points $(0,1,1),(0,1,0)\in\mathbb{R}^3$, then the union of $\phi^m(C_+\cup\beta')(1\leq m\leq n)$ is a connected graph (Figure \ref{fig:--_graphs}, right). Choose a $\phi-$invariant regular neighborhood and let $\Sigma$  be its boundary, then $\phi|_\Sigma$ satisfies Theorem \ref{thm:--} (1) with $g=n/2+1,h=1,b=1,s=0,t=0$.
  \end{example}
  
\subsection{Type $(+,-)$}\label{subsect:example+-}
  \begin{example}\label{ex:+-}
  	$S^3=\mathbb{R}^3\cup\{\infty\}$. $\phi\in O(3)$ has the same matrix as in the last example with $n$ even and $p=1$.
  	
  	(1) Let $\Sigma$ be the sphere $xy-$plane$\,\cup\{\infty\}$. Then $\phi|_\Sigma$ satisfies the conditions in Theorem \ref{thm:+-} with $g=0,h=0,b=0,s=2$ and $\alpha=1$.
  	
  	\begin{figure}[htbp]
  		\centering
  		\setlength{\unitlength}{1bp}%
  		\begin{picture}(353.50, 178.73)(0,0)
  			\put(0,0){\includegraphics{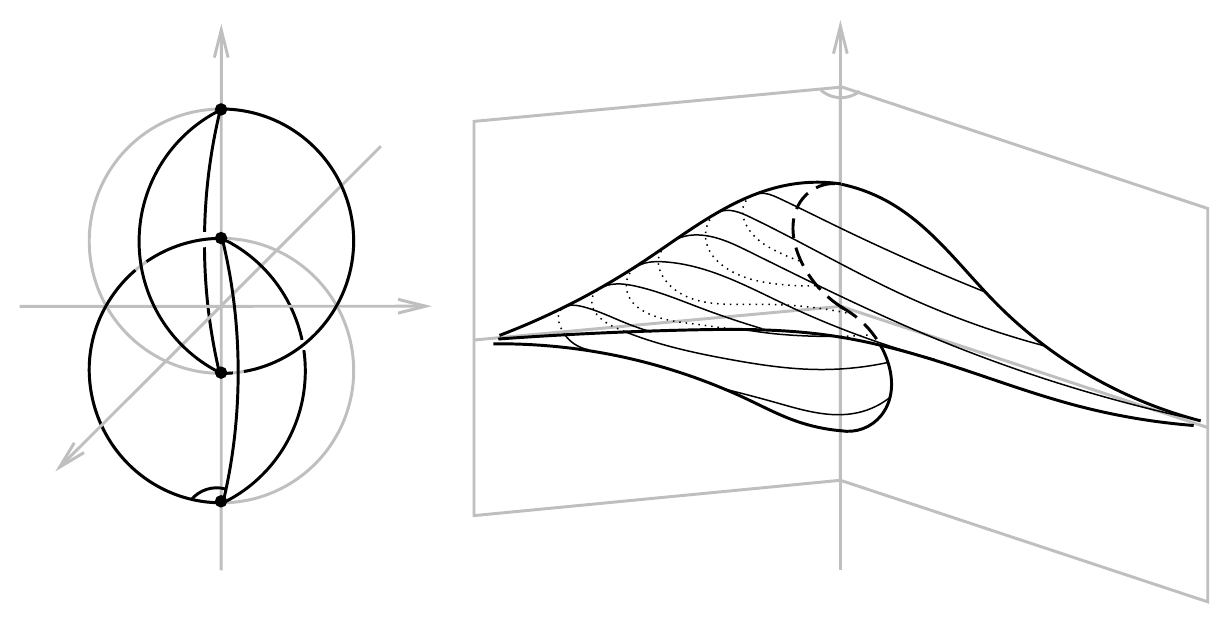}}
  			\put(12.65,46.71){\fontsize{8.83}{10.60}\selectfont $x$}
  			\put(114.54,95.62){\fontsize{8.83}{10.60}\selectfont $y$}
  			\put(66.18,166.00){\fontsize{8.83}{10.60}\selectfont $z$}
  			\put(50.45,42.17){\fontsize{8.83}{10.60}\selectfont $\frac{4\pi}{n}$}
  			\put(87.00,142.45){\fontsize{8.83}{10.60}\selectfont $\Gamma_1$}
  			\put(32.03,32.70){\fontsize{8.83}{10.60}\selectfont $\Gamma_2$}
  			\put(65.42,113.43){\fontsize{8.83}{10.60}\selectfont $1$}
  			\put(52.23,63.72){\fontsize{8.83}{10.60}\selectfont $-1$}
  			\put(230.80,141.58){\fontsize{8.83}{10.60}\selectfont $\frac{2\pi}{n}$}
  			\put(244.81,166.16){\fontsize{8.83}{10.60}\selectfont $z$}
  			\put(243.48,46.57){\fontsize{8.83}{10.60}\selectfont $-2$}
  			\put(243.48,128.19){\fontsize{8.83}{10.60}\selectfont $2$}
  			\put(281.43,101.02){\fontsize{8.83}{10.60}\selectfont $\Sigma$}
  		\end{picture}%
  		\caption{\label{fig:+-_graph}%
  			A $\phi-$invariant graph with two components $(n=6)$ and the equidistance surface in a fundamental domain.}
  	\end{figure}
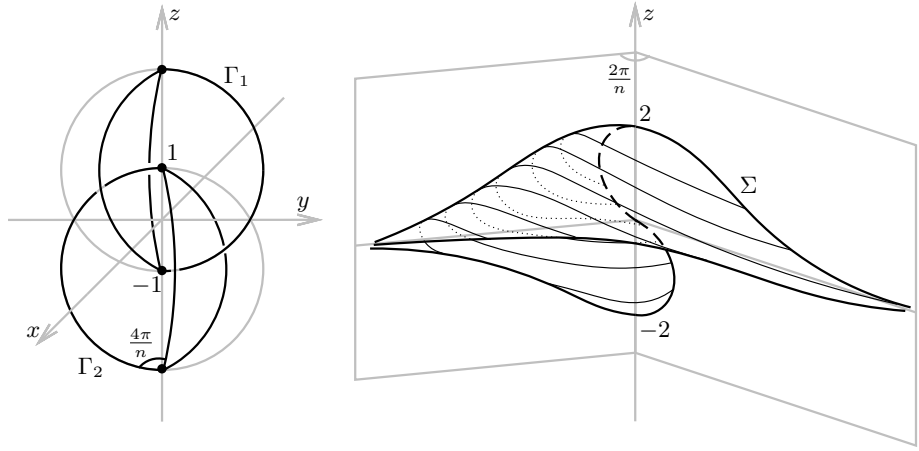
  	
  	(2) Choose an arc
  	\begin{displaymath}
  	\gamma=\{(x,y,z)\in\mathbb{R}^3:y^2+(z-1)^2=4,y\geq 0,x=0\},
  	\end{displaymath}
  	and denote $\Gamma_1=\bigcup\limits_{m=1}^{n/2}\phi^{2m}(\gamma),\,\Gamma_2=\phi(\Gamma_1)$, see Figure \ref{fig:+-_graph}.
  	Let $\Sigma$ be defined by
  	\begin{displaymath}
  	\{X\in\mathbb{R}^3:{\rm dist}(X,\Gamma_1)={\rm dist}(X,\Gamma_2)\}\cup\{\infty\},
  	\end{displaymath}
  	where ${\rm dist}$ is the Euclidean distance in $\mathbb{R}^3$. 
  	The two components of $S^3-\Sigma$ are regular neighborhoods of $\Gamma_1,\Gamma_2$, so $\Sigma$ is a Heegaard surface of genus $n/2-1$. $\phi$ exchanges $\Gamma_1,\Gamma_2$ thus preserves $\Sigma$.
  	If $n\geq 4$, there are three singular points in the quotient orbifold $\Sigma/\phi$, which are provided by the orbits $\{(0,0,0)\},\{(0,0,\pm 2)\},\{\infty\}$; if $n=2$, the orbit $\{(0,0,\pm 2)\}$ gives a regular point so there are only two singular points. Now by the Riemann-Hurwitz formula we see that $\phi|_\Sigma$ satisfies the conditions in Theorem \ref{thm:+-} with $g=n/2-1,h=0,b=0,s=3$ (if $n=2$ then $s=2$) and $\alpha=1$.
  	\end{example}
  
\subsection{Type $(-,+)$}\label{subsect:example-+}
  \begin{example}\label{ex:-+_1}
  	$S^3\subset\mathbb{C}^2,\phi(w_1,w_2)= (w_1e^{\frac{2\pi i}{n/2}},-w_2)$, $n$ is even and $n/2$ is odd. Let $\Sigma$ be the sphere 
  	\begin{displaymath}
  	S^2=\{(w_1,w_2)\in S^3\subset\mathbb{C}^2:{\rm Re}(w_2)=0\}.
  	\end{displaymath}
  	Then $\phi|_\Sigma$ satisfies Theorem \ref{thm:-+} (1) with $g=0,h=0,b=1,s=1$ (if $n=2$ then $s=0$).
  \end{example}

  \begin{example}\label{ex:-+_2}
  	$\Sigma$ is chosen as the union of $\infty$ and a plane in $\mathbb{R}^3$ passing through the origin. Let $\phi$ act on $S^3\subset\mathbb{C}^2$ as $\phi(w_1,w_2)=(-w_1,-w_2)$. 
  	Then $\phi|_\Sigma$ satisfies Theorem \ref{thm:-+} (3) with $g=0,h=1,b=0,s=t=0,p=q=1,l=n=2$. 
  \end{example}

  \begin{example}\label{ex:-+_3}
  	The torus 
  	\begin{displaymath}
  	T=\{(w_1,w_2)\in\mathbb{C}^2:|w_1|=|w_2|=\frac{\sqrt{2}}{2}\}
  	\end{displaymath}
  	splits $S^3\subset\mathbb{C}^2$ into two solid tori:
  	\begin{displaymath}
  	\begin{split}
  	V_1 &=\{(w_1,w_2)\in S^3\subset\mathbb{C}^2:|w_1|\geq \frac{\sqrt{2}}{2}\},\\
  	V_2 &=\{(w_1,w_2)\in S^3\subset\mathbb{C}^2:|w_2|\geq \frac{\sqrt{2}}{2}\}.
  	\end{split}
  	\end{displaymath}
  	The $(2,2)-$torus link 
  	\begin{displaymath}
  	L=\{(\frac{\sqrt{2}}{2}e^{i\theta},\pm\frac{\sqrt{2}}{2}e^{i\theta}):0\leq \theta\leq 2\pi\}
  	\end{displaymath}
  	on $T$ bounds an annulus in $V_1$:
  	\begin{displaymath}
  	\{(r e^{i\theta},\pm\sqrt{1-r^2}e^{i\theta}):0\leq\theta\leq 2\pi,\frac{\sqrt{2}}{2}\leq r\leq 1\},
  	\end{displaymath}
  	which devides $V_1$ into two solid tori $V_{1,1},V_{1,2}$, see Figure \ref{fig:-+_(2,2)-link}. 
  	Fix a positive number $d$, and cut $V_2$ into $2d$ cylinders
  	\begin{displaymath}
  	C_k=\{(w_1,w_2)\in S^2:|w_2|\geq\frac{\sqrt{2}}{2}, \frac{(k-1)\pi}{d}\leq{\rm Arg}(w_2)<\frac{k\pi}{d}\} (1\leq k\leq 2d).
  	\end{displaymath}
  	Define 
  	\begin{displaymath}
  	\begin{split}
  	H_1&=V_{1,1}\cup C_1\cup C_3\cup \cdots\cup C_{2d-1};\\
  	H_2&=V_{1,2}\cup C_2\cup C_4\cup \cdots\cup C_{2d},
  	\end{split}
  	\end{displaymath}
  	and let $\Sigma$ be $\partial H_1=\partial H_2$. For each $k=1,2,\cdots,2d$, $L$ devides the annulus $T\cap\partial C_k$ into two disks, so each of $V_{1,1},V_{1,2}$ intersects $C_k$ at a disk, which implies $H_1,H_2$ are both solid tori and $\Sigma$ is a genus $1$ Heegaard surface of $S^3$.
  	
  	\begin{figure}[htbp]
  		\centering
  		\setlength{\unitlength}{1bp}%
  		\begin{picture}(338.58, 152.85)(0,0)
  			\put(0,0){\includegraphics{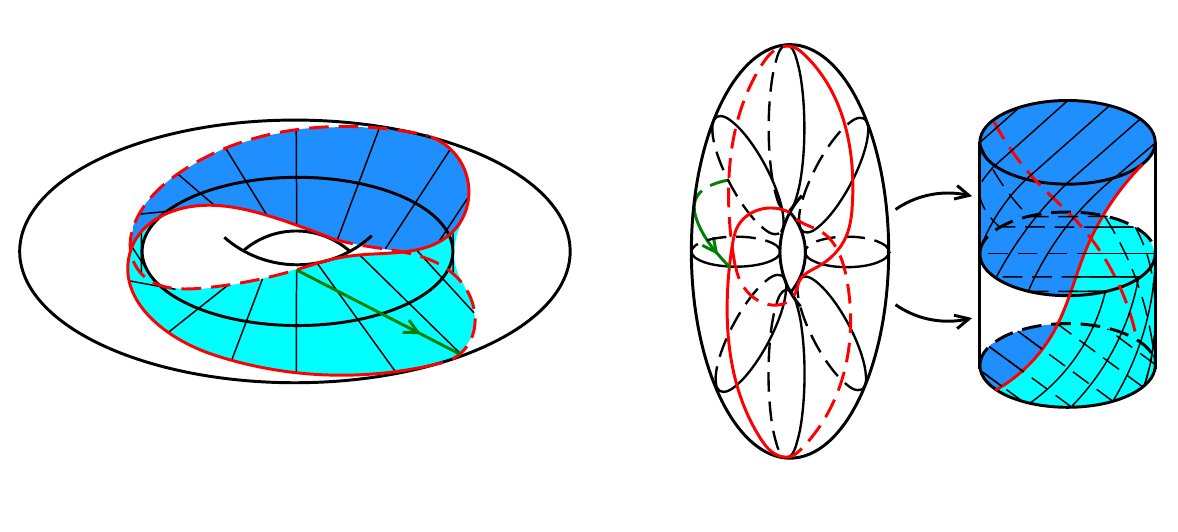}}
  			\put(79.82,28.79){\fontsize{9.96}{11.95}\selectfont $V_1$}
  			\put(122.26,115.57){\rotatebox{360.00}{\fontsize{9.96}{11.95}\selectfont \smash{\makebox[0pt][l]{\textcolor[rgb]{1, 0, 0}{$L$}}}}}
  			\put(220.85,7.81){\fontsize{9.96}{11.95}\selectfont $V_2$}
  			\put(15.65,77.33){\fontsize{9.96}{11.95}\selectfont $V_{1,1}$}
  			\put(139.35,77.33){\fontsize{9.96}{11.95}\selectfont $V_{1,2}$}
  			\put(189.39,103.33){\fontsize{9.96}{11.95}\selectfont $C_1$}
  			\put(254.72,52.08){\fontsize{9.96}{11.95}\selectfont $C_{5}$}
  			\put(302.43,128.28){\fontsize{9.96}{11.95}\selectfont $\Sigma$}
  			\put(254.72,103.67){\fontsize{9.96}{11.95}\selectfont $C_{4}$}
  			\put(189.39,52.03){\fontsize{9.96}{11.95}\selectfont $C_8$}
  			\put(214.80,139.40){\fontsize{9.96}{11.95}\selectfont \textcolor[rgb]{1, 0, 0}{$L$}}
  			\put(90.38,103.20){\rotatebox{360.00}{\fontsize{9.96}{11.95}\selectfont \smash{\makebox[0pt][l]{$S^1$}}}}
  			\put(99.46,69.33){\rotatebox{0.00}{\fontsize{9.96}{11.95}\selectfont \smash{\makebox[0pt][l]{\textcolor[rgb]{0, 0.50196, 0}{$\delta_1$}}}}}
  			\put(190.11,88.02){\rotatebox{0.00}{\fontsize{9.96}{11.95}\selectfont \smash{\makebox[0pt][l]{\textcolor[rgb]{0, 0.50196, 0}{$\delta_2$}}}}}
  		\end{picture}%
  		\caption{\label{fig:-+_(2,2)-link}%
  			A genus $1$ Heegaard surface in $S^3$ $(d=4)$.}
  	\end{figure}
  	
  	(1) Suppose $d$ is odd and consider the automorphism of $S^3$ defined by $\phi(w_1,w_2)=(w_1e^{2\pi i\frac{1}{d}},w_2e^{2\pi i\frac{d+2}{2d}})$. Decompose it as
  	\begin{displaymath}
  	(w_1,w_2)\xrightarrow{\phi_1}(w_1e^{2\pi i\frac{1}{d}},w_2e^{2\pi i\frac{1}{d}})\xrightarrow{\phi_2}(w_1e^{2\pi i\frac{1}{d}},w_2e^{2\pi i\left(\frac{1}{d}+\frac{1}{2}\right)}).
  	\end{displaymath} 
  	On $V_1$, $\phi_1$ preserves each component of $L$ while $\phi_2$ not, thus $\phi$ exchanges them and also the solid tori $V_{1,1},V_{1,2}$. On $V_2$, $\phi$ sends $C_k$ to $C_{k+2+d}$, thus changes the parity of the subscript. Therefore, $\phi$ exchanges $H_1,H_2$ and induces an $f$ on $\Sigma$, which satisfies Theorem \ref{thm:-+} (2) with $n=2d,g=1,h=1,b=1,s=0,l=d$.
  	
  	(2) Suppose $d$ is even and consider the automorphism $\phi$ given by the composition 
  	\begin{displaymath}
  	(w_1,w_2)\rightarrow(w_1e^{2\pi i\frac{1}{2d}},w_2e^{2\pi i\frac{1}{2d}})\rightarrow(w_1e^{2\pi i\frac{1}{2d}},w_2e^{2\pi i\left(\frac{1}{2d}+\frac{1}{2}\right)}).
  	\end{displaymath}
  	Similarly, it induces a periodic map on $\Sigma$ which satisfies Theorem \ref{thm:-+} (3) with $n=2d,g=1,h=2,b=0,s=t=0,p=q=1,l=2d$.
  	$h_1(\phi |_\Sigma),h_2(\phi |_\Sigma)$ must be as in the theorem, though we can compute them directly. Choose a canonical generator system of $\pi_1(\Sigma/\phi)$:
  	\begin{itemize}
  		\item $\delta_1$: represented by the oriented geodesic segment on $\Sigma\cap V_1$ from $(\frac{\sqrt{2}}{2},\frac{\sqrt{2}}{2})$ to $(\frac{\sqrt{2}}{2}e^{2\pi i\frac{1}{2d}},\frac{\sqrt{2}}{2}e^{2\pi i\frac{1+d}{2d}})$;
  		\item $\delta_2$: represented by the oriented curve on $\Sigma\cap C_1$ from $(\frac{\sqrt{2}}{2}e^{2\pi i\frac{d+1}{2d}},\frac{\sqrt{2}}{2}e^{2\pi i\frac{1}{2d}})$ to $(\frac{\sqrt{2}}{2},\frac{\sqrt{2}}{2})$. 
  	\end{itemize}
    By the definition of the epimorphism $\psi$, the deck transformation $\phi^{\psi(\delta_1)}$ sends the initial $(\frac{\sqrt{2}}{2},\frac{\sqrt{2}}{2})$ to the terminal $(\frac{\sqrt{2}}{2}e^{2\pi i\frac{1}{2d}},\frac{\sqrt{2}}{2}e^{2\pi i\frac{1+d}{2d}})$. So $\psi(\delta_1)$ is $1$, and similarly $\psi(\delta_2)=d-1$. Therefore, $h_1(\phi |_\Sigma)=d\in\mathbb{Z}_{2d}$, $h_2(\phi |_\Sigma)=\{\pm 1,\pm(d-1)\}\subset\mathbb{Z}_d$.
	\end{example}
  	
  	\begin{example}\label{ex:-+_4}
  	Replace the $(2,2)-$torus link $L$ in the last example by the $(2,4)-$torus link
  	\begin{displaymath}
  	L'=\{(\frac{\sqrt{2}}{2}e^{i\theta},\pm\frac{\sqrt{2}}{2}e^{2i\theta}):0\leq\theta\leq 2\pi\},
  	\end{displaymath}
  	and define the corresponding handlebodies $V'_{1,1},V'_{1,2},H'_1,H'_2$ similarly. Then we obtain another Heegaard splitting of $S^3$, see Figure \ref{fig:-+_(2,4)-link}. Note that now each of $V'_{1,1},V'_{1,2}$ intersects $C_k$ at two disks, hence the Heegaard surface $\Sigma'=\partial H'_1=\partial H'_2$ has genus $d+1$. 
  	
  	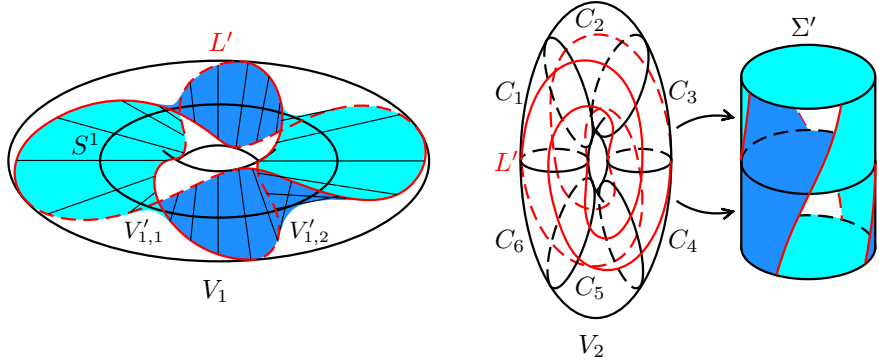
\begin{figure}[htbp]
  		\centering
  		\setlength{\unitlength}{1bp}%
  		\begin{picture}(338.14, 145.82)(0,0)
  			\put(0,0){\includegraphics{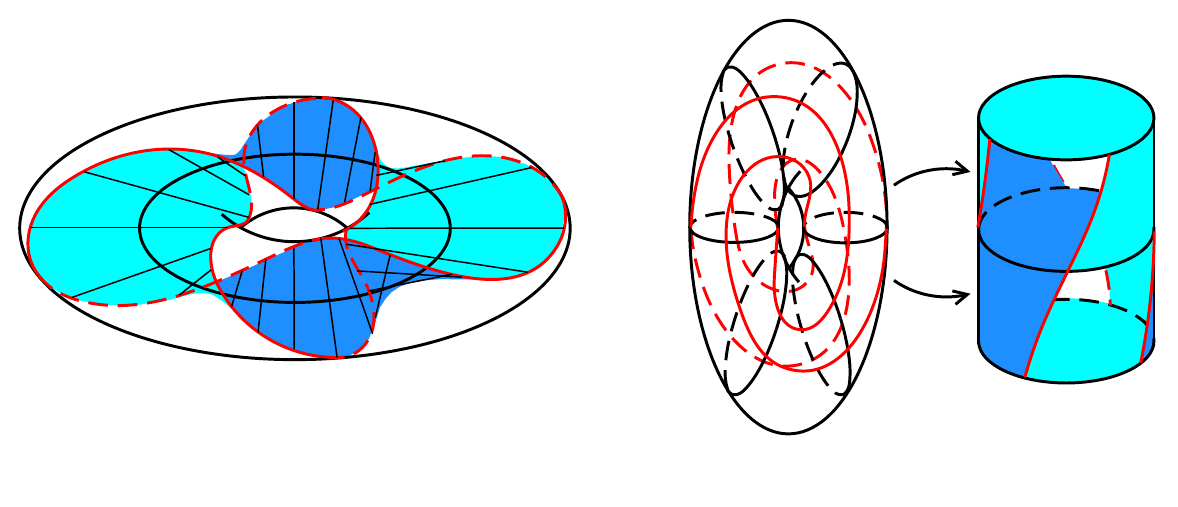}}
  			\put(220.39,7.81){\fontsize{9.96}{11.95}\selectfont $V_2$}
  			\put(188.92,103.33){\fontsize{9.96}{11.95}\selectfont $C_1$}
  			\put(254.25,47.33){\fontsize{9.96}{11.95}\selectfont $C_4$}
  			\put(300.47,127.32){\fontsize{9.96}{11.95}\selectfont $\Sigma'$}
  			\put(254.25,103.67){\fontsize{9.96}{11.95}\selectfont $C_3$}
  			\put(219.18,130.93){\fontsize{9.96}{11.95}\selectfont $C_2$}
  			\put(219.18,29.75){\fontsize{9.96}{11.95}\selectfont $C_5$}
  			\put(188.92,47.39){\fontsize{9.96}{11.95}\selectfont $C_6$}
  			\put(188.59,76.54){\fontsize{9.96}{11.95}\selectfont \textcolor[rgb]{1, 0, 0}{$L'$}}
  			\put(78.54,28.44){\fontsize{9.96}{11.95}\selectfont $V_1$}
  			\put(81.12,121.58){\fontsize{9.96}{11.95}\selectfont \textcolor[rgb]{1, 0, 0}{$L'$}}
  			\put(49.06,51.18){\fontsize{9.10}{10.93}\selectfont $V'_{1,1}$}
  			\put(111.49,51.18){\fontsize{9.10}{10.93}\selectfont $V'_{1,2}$}
  			\put(29.35,82.74){\fontsize{9.96}{11.95}\selectfont $S^1$}
  		\end{picture}%
  		\caption{\label{fig:-+_(2,4)-link}%
  			A genus $d+1$ Heegaard surface in $S^3$ $(d=3)$.}
  	\end{figure}
  		
  	(1) Suppose $d$ is odd and consider the automorphism of $S^3$ defined by the composition 
  	\begin{displaymath}
  	(w_1,w_2)\rightarrow(w_1e^{2\pi i\frac{1}{d}},w_2e^{2\pi i\frac{2}{d}})\rightarrow(w_1e^{2\pi i\frac{1}{d}},w_2e^{2\pi i\left(\frac{2}{d}+\frac{1}{2}\right)}).
  	\end{displaymath}
  	With the Riemann-Hurwitz formula, we see that it induces a periodic map on $\Sigma'$ which satisfies Theorem \ref{thm:-+} (2) with $n=2d,g=d+1,h=2,b=1,s=0,l=d$.
  	
  	(2) Suppose $d$ is odd and consider the composition 
  	\begin{displaymath}
  	(w_1,w_2)\rightarrow(w_1e^{2\pi i\frac{1}{2d}},w_2e^{2\pi i\frac{2}{2d}})\rightarrow(w_1e^{2\pi i\frac{1}{2d}},w_2e^{2\pi i\left(\frac{2}{2d}+\frac{1}{2}\right)}).
  	\end{displaymath}
  	Then the induced map on $\Sigma'$ satisfies Theorem \ref{thm:-+} (3) with $n=2d,g=d+1,h=3,b=0,s=t=0,p=q=1,l=2d$.
	\end{example}

\section{Realizations of the extensions}\label{sect:realizations}

  Suppose a periodic map $f$ on a closed surface satisfies the extension conditions as in one of Theorems \ref{thm:++}, \ref{thm:--}, \ref{thm:+-}, \ref{thm:-+}, we are to prove that it is extendable over $S^3$ in the corresponding type. 
  
  For each case, we will show first that up to conjugacy, the cyclic group $\langle f\rangle$ is determined by the invariants used in the theorems.
  The strategy is to show that fixing a possible collection of such invariants, the followings hold:
  \begin{enumerate}
  	\item The isotropy invariant of $f$ can be normalized. That is to say, there exists some integer $m$ coprime to $n$ (thus $\langle f^m\rangle=\langle f\rangle$) such that the isotropy invariant for $f^m$ has a ``normal" form uniquely determined by the given invariants. The ``normalization''s below are the same. 
  	\item If $h_1(f)$ exists, i.e., $|\Sigma_g/f|$ is non-orientable without boundary, $n/2$ is even and there is no $n/2$ in the isotropy invariant, then $h_1(f)$ can be normalized without changing the isotropy invariant.
  	\item If $h_2(f)$ exists, i.e., $|\Sigma_g/f|$ is non-orientable with genus $h=2$, then $h_2(f)$ can be normalized without changing the isotropy invariant and $h_1(f)$ (if defined).
  \end{enumerate}
  Then we have an integer $m$ coprime to $n$, such that the isotropy invariant of $f^m$ and $h_1(f^m),h_2(f^m)$ (if defined) have a ``normal" form, which is uniquely determined by the given invariants. 
  Moreover, the genus $g$ of the surface can be figured out from the Riemann-Hurwitz formula.
  According to Theorem \ref{thm:classification}, the conjugacy class of $f^m$ is determined. Therefore, $\langle f\rangle=\langle f^m\rangle$ is also determined.
  
  Then it suffices to realize extendable maps with any possible invariants involved. All of them will be established on the examples in Section \ref{sect:examples}. As a consequence of the constructions, Theorem \ref{thm:Heegaard} follows.
  
  The following lemmas will be used to normalize some cases in our discussion.
  
  \begin{lem}\label{lem:power}
  	Suppose $f\in{\rm Aut}(\Sigma_g)$ has period $n$.
  	Let $k,m$ be integers with $km\equiv 1\,({\rm mod}\, n)$.
  	Denote the corresponding epimorphisms for $f$ and $f^m$ by
  	\begin{displaymath}
  	\begin{split}
  	\psi:\pi_1(\Sigma_g/f)\to\langle f\rangle &=\mathbb{Z}_n\\
  	f &\leftrightarrow 1,\\
  	\psi':\pi_1(\Sigma_g/f)\to\langle f^m\rangle &=\mathbb{Z}_n\\
  	f^m &\leftrightarrow 1
  	\end{split}
  	\end{displaymath}
  	respectively. Then for each $\gamma\in\pi_1(\Sigma_g/f)$, we have $\psi'(\gamma)=k\psi(\gamma)$. 
  \end{lem}
  
  \begin{proof}
  	Naturally we have $(f^m)^{\psi'(\gamma)}=f^{\psi(\gamma)}$, so the conclusion follows.
  \end{proof}

  \begin{lem}\label{lem:modulo}
  	Suppose $f\in{\rm Aut}(\Sigma_g)$ has period $n$, and $m$ is a factor of $n$. Let $f_m$ be the order $m$ map induced by $f$ on $|\Sigma_g/f^m|$, and 
  	\begin{displaymath}
  	\begin{split}
  	\psi_m:\pi_1(|\Sigma_g/f^m|/f_m)\to\langle f_m\rangle &=\mathbb{Z}_m\\
  	f_m &\leftrightarrow 1
  	\end{split} 
  	\end{displaymath}
  	be the corresponding epimorphism.
  	Then there is a commutative diagram for $\psi_m,\psi$:
  	\begin{displaymath}
  	\xymatrix{
  		\ar@{}[dr]|{\circlearrowleft}
  		\pi_1(\Sigma_g/f)\ar[r]^-{\psi}\ar[d]_{F_*} & \mathbb{Z}_n\ar[d]^{{\rm mod}\,m} \\
  		\pi_1(|\Sigma_g/f^m|/f_m)\ar[r]^-{\psi_m} & \mathbb{Z}_m
  	}
  	\end{displaymath}
  	where $F_*$ is induced by the natural forgetful map.
  \end{lem}

  \begin{proof}
  	Fix a base point for $\pi_1(\Sigma_g/f)$. It is a regular orbit $X\subset\Sigma_g$ under the $f$-action. Fix $x_0\in X$. For $\gamma\in\pi_1(\Sigma_g/f)$, $\gamma$ can be represented by the image of an oriented curve in $\Sigma_g$ that connects $x_0$ to some $x\in X$. Then $\psi(\gamma)\in\mathbb{Z}_n$ is determined by the equation 
  	\begin{displaymath}
  	f^{\psi(\gamma)}(x_0)=x.
  	\end{displaymath}
  	On the other hand, the oriented curve descends to $|\Sigma_g/f^m|$ and connects $[x_0]$ to $[x]$. It represents $F_*\gamma$ in $\pi_1(|\Sigma_g/f^m|/f_m)$. So $\psi_m(F_*\gamma)\in\mathbb{Z}_m$ is determined by the  equation 
  	\begin{displaymath}
  	(f_m)^{\psi_m(F_*\gamma)}([x_0])=[x].
  	\end{displaymath}
  	Obviously, $\psi(\gamma)$ satisfies the equation, so $\psi_m(F_*\gamma)\equiv\psi(\gamma)\,({\rm mod}\,m)$.
  \end{proof}
  
  \begin{lem}\label{lem:number-theory}
  	(1) Suppose ${\rm gcd}(a,b,c)=1$, then there exists an integer $d$ such that ${\rm gcd}(a+bd,c)=1$.
  	
  	(2) If ${\rm gcd}(m,n)=p$, then there exists an integer $k$, such that ${\rm gcd}(k,n)=1$ and $km\equiv p\,({\rm mod}\,n)$. 
  	
  	(3) Suppose $p,q,p_0,q_0$ are integers. The congruence equation system
  	\begin{displaymath}
  	\left\{
  	\begin{array}{c}
  	x \equiv p_0\,({\rm mod}\,p)\\
  	x \equiv q_0\,({\rm mod}\,q)
  	\end{array} 
  	\right.
  	\end{displaymath}
  	has a solution $x\in\mathbb{Z}$ if and only if $p_0 \equiv q_0\,({\rm mod}\,{\rm gcd}(p,q))$. Moreover, if it has a solution, then the solution is unique modulo the least common multiple ${\rm lcm}(p,q)$. 
  	
  	(4) Suppose $\lambda,\mu\in\mathbb{Z}_n$ have orders $p,q$ respectively, and ${\rm gcd}(p,q)=1,n=pql$. Then there exists a generator $\tau$ of $\mathbb{Z}_n$, such that $\lambda=\tau ql,\mu=\tau pl$. 
  \end{lem}
  
  \begin{proof}
  	(1) Factorize $c$ into $p_1^{s_1}\cdots p_I^{s_I}q_1^{t_1}\cdots q_J^{t_J}$, where $p_i,q_j$ are prime and $p_i|a,q_j\nmid a$ for each $i,j$. Let $d=q_1^{t_1}\cdots q_J^{t_J}$, and suppose ${\rm gcd}(a+bd,c)=c_0$. We only need to show $c_0=1$. 
  	For each $p_i$, $p_i|a,p_i|c,p_i\nmid b,p_i\nmid d$; for each $q_j$, $q_j\nmid a,q_j|d,q_j|c$. So none of $p_i,q_j$ divides $a+bd$, thus $c_0$ must equal $1$.
  	
  	(2) Assume $m=m_0p,n=n_0p$, then ${\rm gcd}(m_0,n_0)=1$. Let $k_1$ satisfy  $k_1m_0\equiv 1\,({\rm mod}\,n_0)$. By (1), we just choose $k=k_1+n_0d(d\in\mathbb{Z})$ such that ${\rm gcd}(k,n)=1$.
  	
  	(3) If $x$ exists, assume $x=p_0+k_1p=q_0+k_2q$, $k_1,k_2\in\mathbb{Z}$. Then $p_0-q_0=-k_1p+k_2q$, so $p_0 \equiv q_0\,({\rm mod}\,{\rm gcd}(p,q))$.
  	
  	Conversely, if $p_0 \equiv q_0\,({\rm mod}\,{\rm gcd}(p,q))$, there are integers $k,a,b$ such that $p_0-q_0=k\cdot{\rm gcd}(p,q)=k(ap+bq)$. Then $x=p_0-kap=q_0+kbq$ is a solution of the equations.
  	
  	If $x,x'$ both satisfy the equations, then $x \equiv x'\,({\rm mod}\,p)$, $x \equiv x'\,({\rm mod}\,q)$, so $x \equiv x'\,({\rm mod}\,{\rm lcm}(p,q))$.
  	
  	(4) Assume $\lambda=\lambda_0ql,\mu=\mu_0pl$ with $1\leq \lambda_0<p,1\leq \mu_0<q$ and ${\rm gcd}(\lambda_0,p)={\rm gcd}(\mu_0,q)=1$. By (3), there exists an integer $\tau_0$ such that $\tau_0\equiv\lambda_0\,({\rm mod}\,p)$ and $\tau_0\equiv\mu_0\,({\rm mod}\,q)$. Moreover, ${\rm gcd}(\tau_0,pq)=1$. By (1), we just choose $\tau=\tau_0+pqd(d\in\mathbb{Z})$ such that ${\rm gcd}(\tau,n)=1$. 
  \end{proof}

\subsection{Type $(+,+)$}\label{subsect:realization++}
  \begin{prop}\label{prop:conjugacy++}
  	Suppose $f$ satisfies the conditions in Theorem \ref{thm:++}, then the conjugacy class of $\langle f\rangle$ is uniquely determined by $n,h,s,t,p,q$. 
  \end{prop}
  
  \begin{proof}
  	We fix $n,h,s,t,p,q$ and follow the strategy at the beginning of this section. 
  	
  	Assume $n=pql$ for some positive integer $l$. By Lemma \ref{lem:number-theory} (4), the isotropy invariant is 
  	\begin{displaymath}
  	(\underbrace{\tau ql,\cdots,\tau ql}_{t},\underbrace{-\tau ql,\cdots,-\tau ql}_{t},\underbrace{\tau pl,\cdots,\tau pl}_{s/2-t},\underbrace{-\tau pl,\cdots,-\tau pl}_{s/2-t}),
  	\end{displaymath}
  	where $\tau$ is a generator of $\mathbb{Z}_n$.
  	By Lemma \ref{lem:power}, the isotropy invariant for $f^\tau$ is
  	\begin{displaymath}
  	(ql,\cdots,ql,-ql,\cdots,-ql,pl,\cdots,pl,-pl,\cdots,-pl).
  	\end{displaymath}
  	Moreover, $g$ can be figured out from the Riemann-Hurwitz formula
  	\begin{displaymath}
  	2-2g=n\left(2-2h-2t(1-\frac{1}{p})-(s-2t)(1-\frac{1}{q})\right).
  	\end{displaymath}
  	According to Theorem \ref{thm:classification} (1), $f^\tau$ is determined up to conjugacy and thus so is $\langle f\rangle=\langle f^\tau\rangle$.
  \end{proof}

  Maps extendable in type $(+,+)$ and the corresponding embedded surfaces have been described in \cite{NWW}. We introduce another approach to constructing them as follows, and the strategy will be applied later to obtain all extendable maps in other three types from basic examples in Section \ref{sect:examples}.
  
  \begin{example}\label{ex:++general}
  	Let $\phi$ be ${\rm id}_{S^3}$ and $\Sigma$ be a trivially embedded sphere in $S^3$ which does not intersect the two circles $S^1,Z$, see Figure \ref{fig:modification}. 
  	Then of course $\phi|_\Sigma$ satisfies the conditions in Theorem \ref{thm:++} with parameters $g=h=0,s=t=0,n=p=q=1$. 
  	We can modify $\Sigma$ as in the figure to get another $\phi$-invariant surface $\Sigma'$, which has larger genus and more intersection points with $S^1, Z$. 
  	Then we consider a branched covering $S^3\to S^3$ with branch sets $S^1\cup Z$ (both upstairs and downstairs). Lifting $\Sigma'$ from the downstairs $S^3$ to the upstairs $S^3$, we obtain a surface $\tilde{\Sigma}$. Let $\tilde{\phi}$ be a periodic automorphism of $S^3$ such that the quotient map $S^3\to |S^3/\tilde{\phi}|\cong S^3$ is exactly the branched covering.
  	By this means, the extension for a map in Theorem \ref{thm:++} with any possible parameters $\tilde{n}=\tilde{p}\tilde{q},\tilde{h},\tilde{s},\tilde{t},\tilde{p},\tilde{q}$ can be realized as $(\tilde{\phi}|_{\tilde{\Sigma}},\tilde{\phi})$.
  	
  	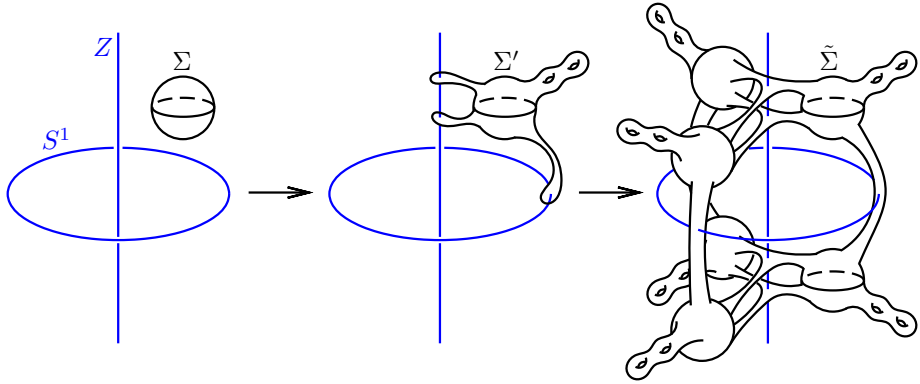
\begin{figure}[htbp]
  		\centering
  		\setlength{\unitlength}{1bp}%
  		\begin{picture}(353.76, 152.96)(0,0)
  			\put(0,0){\includegraphics{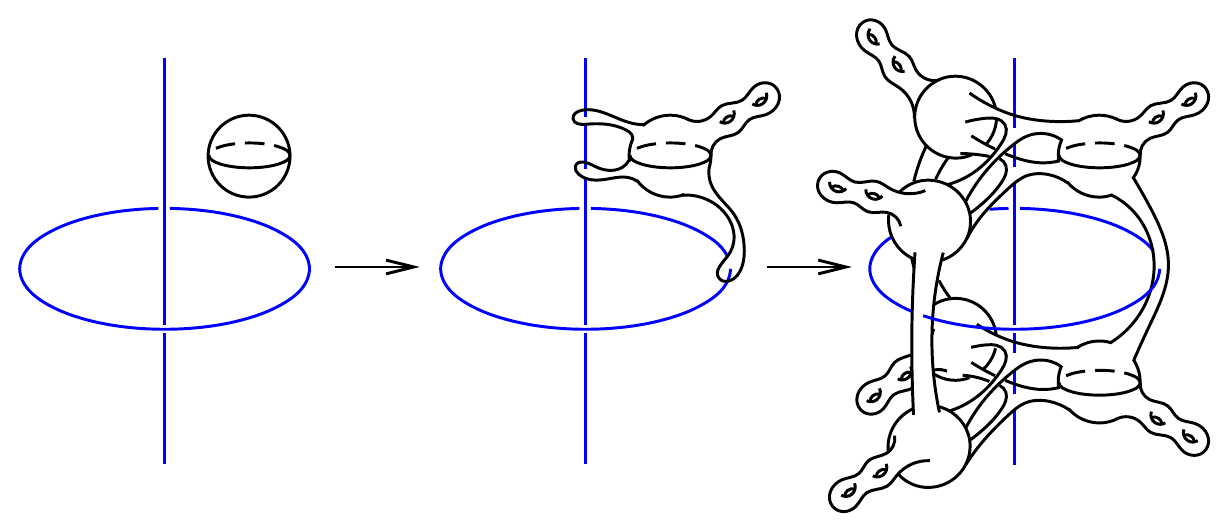}}
  			\put(37.85,127.75){\fontsize{9.96}{11.95}\selectfont \textcolor[rgb]{0, 0, 1}{$Z$}}
  			\put(18.57,92.37){\fontsize{9.96}{11.95}\selectfont \textcolor[rgb]{0, 0, 1}{$S^1$}}
  			\put(67.77,122.18){\fontsize{9.96}{11.95}\selectfont $\Sigma$}
  			\put(188.67,122.22){\fontsize{9.96}{11.95}\selectfont $\Sigma'$}
  			\put(311.57,122.32){\fontsize{9.96}{11.95}\selectfont $\tilde\Sigma$}
  		\end{picture}%
  		\caption{\label{fig:modification}%
  			Constructing new extendable maps with surgeries  ($\tilde{n}=6,\tilde{h}=2,\tilde{s}=6,\tilde{t}=1,\tilde{p}=2,\tilde{q}=3$).}
  	\end{figure}
  \end{example}
  
  Generally, suppose we have a periodic automorphism $\phi$ of $S^3$ and a $\phi$-invariant Heegaard surface $\Sigma$, such that the induced map $\phi|_\Sigma$ satisfies the conditions in Theorem \ref{thm:++} with parameters $g,h,s,t,n,p,q$.
  We can make modifications in a fundamental domain for the $\phi$-action on $(S^3,\Sigma)$, and extend them $\phi$-equivariantly. 
  In other words, we apply the following surgeries, and then lift the surface in $|S^3/\phi|$ to $\Sigma'$ in $S^3$:
  \begin{itemize}
  	\item \textbf{genus surgery}: modifying the quotient surface $|\Sigma/\phi|$ in $|S^3/\phi|$ to increase the genus by adding a ``handle''. 
  	\item \textbf{singular surgery}: modifying the quotient surface $|\Sigma/\phi|$ in $|S^3/\phi|$ to enlarge the intersection sets $|\Sigma/\phi|\cap|S^1/\phi|, |\Sigma/\phi|\cap|Z/\phi|$ by an isotopy which seems like ``stretching out hands to grip the axes''. 
  \end{itemize}
  After that, take a branched covering $S^3\to S^3$ with branch sets $S^1\cup Z$ (both upstairs and downstairs). Lifting $\Sigma'$ and $\phi$ from the downstairs $S^3$ to the upstairs $S^3$, we obtain a surface $\tilde{\Sigma}$ and a periodic automorphism $\tilde{\phi}$. In other words, we conduct the following surgery:
  \begin{itemize}
  	\item \textbf{branch surgery}: suppose $\phi$ acts on $S^3\subset\mathbb{C}^2$ as
  	\begin{displaymath}
  	\phi(w_1,w_2)=(w_1e^{2\pi i\frac{m_1}{n}},w_2e^{2\pi i\frac{m_2}{n}})
  	\end{displaymath}
  	with ${\rm gcd}(n,m_1,m_2)=1$. 
  	Let $\tilde{\phi}:S^3\to S^3$ be the map defined by
  	\begin{displaymath}
  	\tilde{\phi}(w_1,w_2)\mapsto(w_1e^{2\pi i\frac{m_1}{nn_1}},w_2e^{2\pi i\frac{m_2}{nn_2}})
  	\end{displaymath}
  	with ${\rm gcd}(m_1,n_1)={\rm gcd}(m_2,n_2)={\rm gcd}(n_1,n_2)=1$. 
  	Then the $\tilde{\phi}^n$-action induces a branched covering map
  	\begin{displaymath}
  	\begin{matrix}
  	\omega: & S^3 &\to &|S^3/(\tilde{\phi})^n|\cong S^3\\
  	& (w_1,w_2) &\mapsto &(|w_1|e^{i\cdot n_1\cdot {\rm Arg}(w_1)},|w_2|e^{i\cdot n_2\cdot {\rm Arg}(w_2)}).
  	\end{matrix}
  	\end{displaymath}
  	For a $\phi$-invariant surface $\Sigma'\subset |S^3/(\tilde{\phi})^n|\cong S^3$, denote $\tilde{\Sigma}=\omega^{-1}(\Sigma')$.
  \end{itemize}
  As $|\tilde{\Sigma}/\tilde{\phi}|\cong|\Sigma'/\phi|$, we have constructed a new extendable map $\tilde{\phi}|_{\tilde{\Sigma}}$ which still satisfyies the conditions in Theorem \ref{thm:++}. Denote the corresponding parameters by $\tilde{g},\tilde{n},\tilde{h},\tilde{s},\tilde{t},\tilde{p},\tilde{q}$. We see that $\tilde{h}-h$ is the genus of the added ``handle'' in the genus surgery; $2\tilde{t},\tilde{s}-2\tilde{t}$ are the cardinalities of $|\Sigma'/\phi|\cap|S^1/\phi|, |\Sigma'/\phi|\cap|Z/\phi|$ respectively; $\tilde{p}=pn_2,\tilde{q}=qn_1$; $\tilde{n}=nn_1n_2$; and $\tilde{g}$ is determined by the Riemann-Hurwitz formula. 
  Moreover, $\tilde{\Sigma}$ can be chosen as a Heegaard surface in $S^3$.
  
  Now from Example \ref{ex:++} we have an extendable map with parameters $n>1,h=1,s=t=0,p=q=1$. Apply the three surgeries. Similarly we can obtain each extendable map which has possible parameters $\tilde{g},\tilde{n},\tilde{h},\tilde{s},\tilde{t},\tilde{p},\tilde{q}$ with $\tilde{n}>\tilde{p}\tilde{q}$. Note that in this case $\tilde{h}$ must be non-zero for otherwise the corresponding epimorphism $\tilde{\psi}$ can not be surjective. 
  Also, during the surgeries we can always make $\tilde{\Sigma}$ a Heegaard surface of $S^3$.
  
  So far we have proved that for a surface map satisfying the conditions in Theorem \ref{thm:++}, it is extendable in type $(+,+)$. Moreover, according to its parameters, we can construct such a map together with an extension as above. For a direct (and essentially the same) description, readers may refer to \cite{NWW}.
  
\subsection{Type $(-,-)$}\label{subsect:realization--}
  \begin{prop}\label{prop:conjugacy--}
  	(1) Suppose $f$ satisfies Theorem \ref{thm:--} (1), then the conjugacy class of $\langle f\rangle$ is uniquely determined by $n,h,b,s,t$. 
  	
  	(2) Suppose $f$ satisfies Theorem \ref{thm:--} (2) or (3), then the conjugacy class of $\langle f\rangle$ is uniquely determined by $n,h,s$.
  \end{prop}
  
  \begin{proof}
  	(1) Fix $n,h,b,s,t$. Just as in the proof of Proposition \ref{prop:conjugacy++}, we may figure out $g$ from the Riemann-Hurwitz formula, and normalize the isotropy invariant to make $\alpha=1\in\mathbb{Z}_n$.
  	Then by Theorem \ref{thm:classification} (1), $\langle f\rangle$ is determined up to conjugacy. 
  	
  	(2) Fix $n,h,s$. Similarly we figure out $g$ and normalize $\alpha$ to be $1\in\mathbb{Z}_n$.
  	In the following cases, $h_1(f)$ does not exist:
  	\begin{itemize}
  		\item $n/2$ is odd;
  		\item $n=4,s\neq 0$, thus we have $2\alpha=2=n/2$ in the isotropy invariant.
  	\end{itemize}
    And in the other cases, $h_1(f)$ is fixed by the condition (3) in Theorem \ref{thm:--}. So it suffices to normalize $h_2(f)$ for the case $h=2$. 
  	If $s>0$, $h_2(f)$ must be $\{\pm 1,\pm 1\}\subset\mathbb{Z}_2$ as ${\rm gcd}(2\alpha,n)=2$, then $\langle f\rangle$ is determined. Assume $h=2,s=0$ below, thus by definition and Lemma \ref{lem:modulo}, we have
  	\begin{displaymath}
  	h_2(f)=\{\pm\psi(\delta_1),\pm\psi(\delta_2)\}\subset\mathbb{Z}_d,\,d={\rm gcd}(\psi(\delta_1)+\psi(\delta_2),n).
  	\end{displaymath}
  	
  	If $n/2$ is odd, from the relation
  	\begin{displaymath}
  	2(\psi(\delta_1)+\psi(\delta_2))=\psi(\delta_1^2\delta_2^2)=\psi(1)=0\in\mathbb{Z}_n
  	\end{displaymath}
  	we see that $\psi(\delta_1)+\psi(\delta_2)=0\in\mathbb{Z}_n$ for $\psi(\delta_1),\psi(\delta_2)$ are both odd. 
  	If $n/2$ is even, with the assumptions in the theorem we also have $\psi(\delta_1)+\psi(\delta_2)=h_1(f)=0$. So $d=n$ holds for either case.
  	$\psi$ is surjective, so $\psi(\delta_1),\psi(\delta_2)$ are both coprime with $n$. Choose $m=\psi(\delta_1)$, then $f^m$ normalizes the invariant $h_2$. In fact, let $\psi'$ be the corresponding epimorphism for $f^m$, then for each $\gamma\in\pi_1(\Sigma_g/f)$, by Lemma \ref{lem:power} we have $\psi'(\delta_1)=1, \psi'(\delta_2)=-1$. Therefore,
  	\begin{displaymath}
  	h_2(f^m)=\{\pm\psi'(\delta_1),\pm\psi'(\delta_2)\}=\{\pm 1,\pm(n-1)\}\subset\mathbb{Z}_n.
  	\end{displaymath}
  	Moreover, $f,f^m$ have no isotropy invariant, and $h_1(f^m)$, if exists, equals $\psi'(\delta_1)+\psi'(\delta_2)=0=h_1(f)$.
  	As a consequence, the conjugacy class of $\langle f\rangle=\langle f^m\rangle$ is uniquely determined.
  \end{proof}
  
  For type $(-,-)$, most extendable maps have been constructed in \cite{C2}. Also, it is convenient to realize them from Example \ref{ex:--} with surgeries. 
  
  \begin{example}\label{ex:--general}
  Suppose $f$ satisfies the condition (1) in Theorem \ref{thm:--} with parameters $n,h=2,b=6,s=5,t=4$. Then it can be realized from Example \ref{ex:--} (1)(ii). In fact, the constructed surface $\Sigma_g$ can also be directly described as follows (just with a little deformation). 
  
  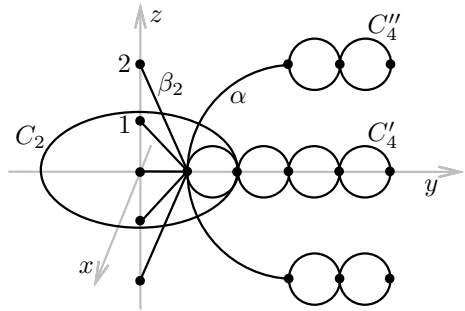
\begin{figure}[htbp]
  	\centering
  	\setlength{\unitlength}{1bp}%
  	\begin{picture}(181.36, 127.68)(0,0)
  		\put(0,0){\includegraphics{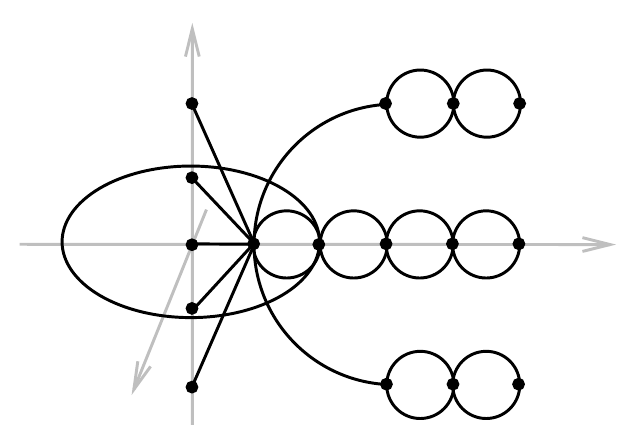}}
  		\put(47.03,72.15){\fontsize{9.96}{11.95}\selectfont $1$}
  		\put(46.95,95.22){\fontsize{9.96}{11.95}\selectfont $2$}
  		\put(58.78,114.23){\fontsize{9.96}{11.95}\selectfont $z$}
  		\put(32.16,19.37){\fontsize{9.96}{11.95}\selectfont $x$}
  		\put(162.43,50.24){\fontsize{9.96}{11.95}\selectfont $y$}
  		\put(8.24,68.07){\fontsize{9.96}{11.95}\selectfont $C_2$}
  		\put(141.05,69.18){\fontsize{9.96}{11.95}\selectfont $C'_4$}
  		\put(141.05,109.91){\fontsize{9.96}{11.95}\selectfont $C''_4$}
  		\put(89.15,83.57){\fontsize{9.96}{11.95}\selectfont $\alpha$}
  		\put(61.81,88.22){\fontsize{9.96}{11.95}\selectfont $\beta_2$}
  	\end{picture}%
  	\caption{\label{fig:graph}%
  		The graph $\Gamma$ ($I=2,J=1,K=3,L=2$).}
  \end{figure}
  
  See Figure \ref{fig:graph}. Denote
  \begin{align*}
  C_j &=\{(x,y,z)\in\mathbb{R}^3:z=0,x^2+y^2=j^2\},\\
  C'_k &=\{(x,y,z)\in\mathbb{R}^3:x=0,(y-(k+\frac{1}{2}))^2+z^2=\frac{1}{4}\},\\
  C''_l &=\{(x,y,z)\in\mathbb{R}^3:x=0,(y-(l+\frac{1}{2}))^2+(z-2)^2=\frac{1}{4}\},\\
  \alpha &=\{(x,y,z)\in\mathbb{R}^3:x=0,y\leq 3,z\geq 0,(y-3)^2+z^2=4\},
  \end{align*} 
  and let $\beta_i$ be the line segment connecting the points $(0,0,i),(0,1,0)$. 
  Choose $\phi\in{\rm Aut}(S^3)$ as in Example \ref{ex:--} (1)(ii), i.e., the composition of a
  $\frac{2\pi}{n/2}$-rotation around the $z$-axis and a reflection across
  the $xy$-plane.
  For non-negative integers $I,J,K,L$, choose $\Sigma$ to be the boundary of a $\phi$-invariant regular neighborhood of the graph
  \begin{displaymath}
  \Gamma_{I,J,K,L}\triangleq\bigcup_{m=1}^{n}\phi^m\left(\alpha\cup(\bigcup_{i=0}^{I}\beta_i)\cup(\bigcup_{j=1}^{J}C_{j+1})\cup(\bigcup_{k=1}^{J+K}C'_k)\cup(\bigcup_{l=1}^{L}C''_{l+2})\right).
  \end{displaymath} 
  $\Sigma$ is a connected surface of genus
  \begin{displaymath}
   g=\frac{n}{2}(2I+2J+K+2L)-2I.
  \end{displaymath}
  Then $\phi|_\Sigma$ satisfies the conditions (1) in Theorem \ref{thm:--} with
  $h=L,\,b=2J+K+1,\,s=2I+1,\,t=I+J+1$. Particularly, the case $I=2,J=1,K=3,L=2$ in the figure gives $h=2,b=6,s=5,t=4$.
  \end{example}
  
  We give an explanation for the idea of the example. Beginning with the basic example, we have an initial surface $\Sigma$ which is isotopic to the boundary of a $\phi$-invariant regular neighborhood of the initial graph $\Gamma_{0,0,0,0}$. The parameters for the surface map $\phi|_\Sigma$ are $n,h=0,b=1,s=t=1$.
  We hope to enlarge the parameters by modifying the quotient space pair $(|S^3/\phi|,|\Sigma/\phi|)$. 
  Just as in type $(+,+)$, we apply genus surgery and singular surgery to enlarge $h,s$. In Figure \ref{fig:graph} the surgeries are realized by adding $C_l''$ and $\beta_i(i\geq 1)$, i.e., enlarging the parameters $L,I$ for the graph. For boundaries, we need two more surgeries:
  \begin{itemize}
  	\item \textbf{Boundary surgery I}: when $\partial|S^3/\phi|\neq\emptyset$, applying a connected sum on $|\Sigma/\phi|$ with a properly embedded disk in $|(S^3-Z)/\phi|$, we can increase $b$ by $1$ without changing $t$. In the example, we add $C_k'$, i.e., enlarge $K$, to accomplish it.
  	\item \textbf{Boundary surgery II}: when $\partial|S^3/\phi|\neq\emptyset$, applying a connected sum on $|\Sigma/\phi|$ with a properly embedded annulus $A$ in $|(S^3-Z)/\phi|$ such that $\partial A$ surrounds $|Z/\phi|$, we can increase $b,t$ by $2,1$ respectively. In the example, we add $C_j$, i.e., enlarge $J$, to accomplish it.
  \end{itemize}
  
  The same trick works for the remaining cases in type $(-,-)$. Table \ref{table:--} lists the constructions of all extendable maps, where the first three cases are not disjoint, and the fourth case has been discussed in detail as above. 
  Note that if Theorem \ref{thm:--} (1) holds with $s=2t=0$, $h$ must be positive for otherwise $\psi$ can not be surjective. And if Theorem \ref{thm:--} (3) holds, $h,s$ must have the same parity, according to the equation
  \begin{displaymath}
  2(\sum_{i=1}^{h}\psi(\delta_i))+\sum_{k=1}^{s}\psi(\xi_k)\equiv 0\,({\rm mod}\,n)
  \end{displaymath}
  and $2\psi(\delta_i)\equiv\psi(\xi_k)\equiv 2\,({\rm mod}\,4)$.
  
  \begin{table}[htbp]
  	\caption{Extendable periodic maps in type $(-,-)$.}
  	\label{table:--}
  	\centering
  	\begin{tabularx}{300pt}{l|X}
  		\toprule
  		\multicolumn{1}{c}{Case (in Theorem \ref{thm:--})} & \multicolumn{1}{c}{Initial example (in Example \ref{ex:--})}\\
  		\midrule
  		(1) with $s$ even and $t=0$ & (4) with $h=1,b=1,s=0,t=0$ \\
  		\midrule
  		(1) with $s=2t>0$ & (3) with $h=0,b=1,s=2,t=1$\\
  		\midrule
  		(1) with $s$ even and $t>\frac{s}{2}$ & (2)(ii) with
  		$h=0,b=2,s=0,t=1$\\
  		\midrule
  		(1) with $s$ odd & (1)(ii) with
  		$h=0,b=1,s=1,t=1$ \\
  		\midrule
  		(2) with $s$ even & (2)(i) with $n/2$ odd and $h=2,s=0$\\
  		\midrule
  		(2) with $s$ odd & (1)(i) with $n/2$ odd and $h=1,s=1$\\
  		\midrule
  		(3) with $s$ even & (2)(i) with $n/2$ even and $h=2,s=0$\\
  		\midrule
  		(3) with $s$ odd & (1)(i) with $n/2$ even and $h=1,s=1$\\
  		\bottomrule
  	\end{tabularx}
  \end{table}

\subsection{Type $(+,-)$}\label{subsect:realization+-}
  \begin{prop}\label{prop:conjugacy+-}
  	Suppose $f$ satisfies the conditions in Theorem \ref{thm:+-}, then the conjugacy class of $\langle f\rangle$ is uniquely determined by $n,h,s$. 
  \end{prop}

  This proposition is an immediate consequence of Theorem \ref{thm:classification} (1). So we just turn to the realizations of all extendable maps in type $(+,-)$.
  If $s$ is even, take Example \ref{ex:+-} (1) with $h=0,s=2$; and if $s$ is odd, take Example \ref{ex:+-} (2) with $h=0,s=3$. Similarly, we can apply genus surgery and singular surgery to increase $h,s$ in either case and complete the construction. 

\subsection{Type $(-,+)$}\label{subsect:realization-+}
  \begin{example}\label{ex:-+Klein}
  	We first construct a family of extendable maps such that the quotient surfaces are Klein bottles. 
  	
  	In Example \ref{ex:-+_3} (2), we have an automorphism $\phi$ acting on $(S^3,\Sigma)$ as 
  	\begin{displaymath}
  	\phi(w_1,w_2)=(w_1e^{2\pi i\frac{1}{l}},w_2e^{2\pi i\frac{l/2+1}{l}})
  	\end{displaymath}
  	with $l\equiv 0 \,({\rm mod}\,4)$. 
  	We are now to do branch surgery on it. As the branch axis $S^1$ is on $\Sigma$, generally $\Sigma$ does not lift to a surface. Though, we can disturb $\Sigma$ $\phi$-equivariantly to get over it. For example, take the solid tori $V_1,V_2$ as in Example \ref{ex:-+_3}, and let $V_0$ be a thinner solid torus in $V_1$:
  	\begin{displaymath}
  	V_0=\{(w_1,w_2)\in S^3\subset\mathbb{C}^2:|w_2|\leq \frac{1}{2}\}.
  	\end{displaymath}
  	We disturb $\Sigma$ in $V_0$ so that it looks like the same as in $V_2$. That is to say, we
  	define 
  	\begin{displaymath}
  	\begin{array}{cccc}
  	\tau : & V_2 &\to & V_0\\
  	& (w_1,w_2) &\mapsto & (\frac{\sqrt{1-|w_1|^2/2}}{|w_2|}w_2,\frac{\sqrt{2}}{2}w_1),
  	\end{array}
  	\end{displaymath}
  	and then replace $\Sigma$ with the surface
  	\begin{displaymath}
  	\hat{\Sigma}=(\Sigma-\Sigma\cap V_0)\cup \tau(\Sigma\cap V_2).
  	\end{displaymath}
  	
  	Suppose $p,q$ are coprime positive integers. Without loss of generality, assume $p$ is odd. 
  	Choose the smallest $m_0\in\mathbb{N}$ such that
  	\begin{displaymath}
  	\left\{\begin{array}{l}
  	m_0\equiv l/2+1 \,({\rm mod}\,l),\\
  	{\rm gcd}(m_0,p)=1.
  	\end{array}
  	\right.
  	\end{displaymath}
  	By Lemma \ref{lem:number-theory} (1), such $m_0$ exists.
  	Let $\phi':S^3\to S^3$ be the map defined by
  	\begin{displaymath}
  	\phi'(w_1,w_2)=(w_1e^{2\pi i\frac{1}{ql}},w_2e^{2\pi i\frac{m_0}{pl}}).
  	\end{displaymath}
  	Then $(\phi')^l$ induces a branched covering map
  	\begin{displaymath}
  	\begin{matrix}
  	\omega: & S^3 &\to &|S^3/(\phi')^l|\cong S^3\\
  	& (w_1,w_2) &\mapsto &(|w_1|e^{i\cdot q\cdot {\rm Arg}(w_1)},|w_2|e^{i\cdot p\cdot {\rm Arg}(w_2)}).
  	\end{matrix}
  	\end{displaymath}
  	For the $\phi$-invariant surface $\hat{\Sigma}\subset |S^3/(\phi')^l|\cong S^3$, denote $\Sigma'=\omega^{-1}(\hat{\Sigma})$. Obviously, $|\Sigma'/\phi'|\cong|\hat{\Sigma}/\phi|\cong|\Sigma/\phi|$, so the restriction of $\phi'$ on $\Sigma'$ is an extendable map with $n=pql,h=2,s=2,t=1$.
  	Besides, we can also use singular surgery to make $t>1$ or $s-t>1$. 
  	
  	Denote the extendable surface map $\phi' |_{\Sigma'}$ by $f'$. Now we compute its isotropy invariant and $h_1(f'), h_2(f')$. 
  	Denote the epimorpshim for $f'$ by $\psi'$, and let $i_*:\pi_1(\Sigma'/f')\to\pi_1(S^3/\phi')$ be the homomorphism induced by the inclusion. Then there is a commutative diagram with the two rows exact: 
  	\begin{displaymath}
  	\xymatrix{
  		1\ar[r] & \pi_1(\Sigma')\ar[r] & \pi_1(\Sigma'/f')\ar[r]^-{\psi'}\ar[d]_{i_*} & \mathbb{Z}_{pql}\ar[d]^{f'\leftrightarrow 1\leftrightarrow\phi'}_=\ar[r] & 1\\
  		1\ar[r] & \pi_1(S^3)\ar[r] & \pi_1(S^3/\phi')\ar[r] & \mathbb{Z}_{pql}\ar[r] & 1.
  	}
  	\end{displaymath}
  	With an abuse of notation, denote the epimorphism for the $\phi'$-action on $S^3$ by $\psi'$ too:
  	\begin{displaymath}
  	\begin{split}
  	\psi':\pi_1(S^3/\phi')\to\langle \phi'\rangle &=\mathbb{Z}_{pql}\\
  	\phi' &\leftrightarrow 1.
  	\end{split}
  	\end{displaymath} 
  	The $\phi'$-action on $S^3\subset\mathbb{C}^2$ has a fundamental domain $F$ which is the convex part bounded by the disks ${\rm Arg}(w_1)=0$, ${\rm Arg}(w_1)=\frac{2\pi}{ql}$, ${\rm Arg}(w_2)=-\frac{\pi}{p}$ and ${\rm Arg}(w_2)=\frac{\pi}{p}$, see Figure \ref{fig:fundamental-domain}. $|S^3/\phi'|$ is homeomorphic to a lens space $L(l,m_0)\cong L(l,l/2+1)$, which can be obtained from $F$ by gluing its boundary.  
  	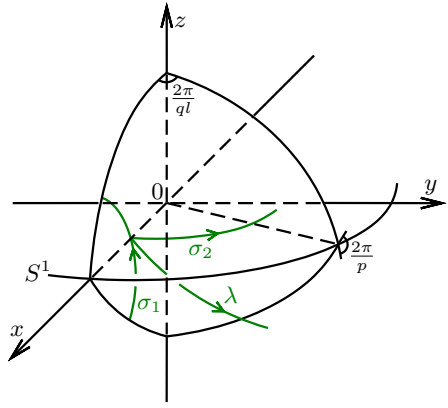
\begin{figure}[htbp]
  		\centering
  		\setlength{\unitlength}{1bp}%
  		\begin{picture}(174.02, 160.44)(0,0)
  			\put(0,0){\includegraphics{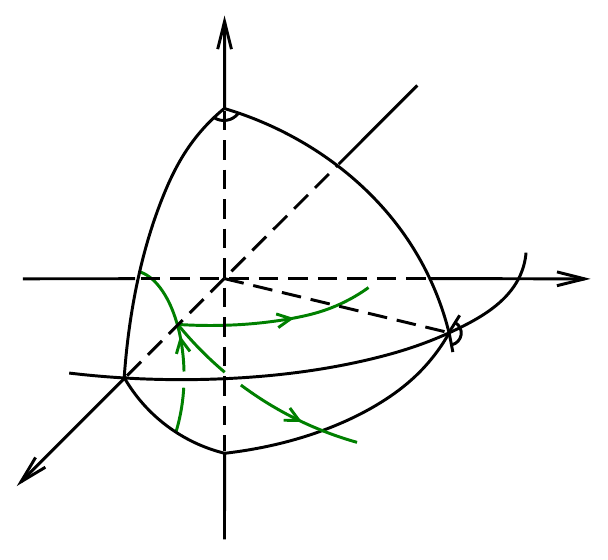}}
  			\put(5.67,30.73){\fontsize{8.83}{10.60}\selectfont $x$}
  			\put(161.77,85.57){\fontsize{8.83}{10.60}\selectfont $y$}
  			\put(67.43,147.01){\fontsize{8.83}{10.60}\selectfont $z$}
  			\put(65.66,117.90){\fontsize{8.83}{10.60}\selectfont $\frac{2\pi}{ql}$}
  			\put(132.94,58.48){\fontsize{8.83}{10.60}\selectfont $\frac{2\pi}{p}$}
  			\put(58.94,82.40){\fontsize{8.83}{10.60}\selectfont $0$}
  			\put(11.03,51.21){\fontsize{8.83}{10.60}\selectfont $S^1$}
  			\put(53.50,40.72){\fontsize{8.83}{10.60}\selectfont \textcolor[rgb]{0, 0.50196, 0}{$\sigma_1$}}
  			\put(72.92,61.03){\fontsize{8.83}{10.60}\selectfont \textcolor[rgb]{0, 0.50196, 0}{$\sigma_2$}}
  			\put(86.09,43.10){\fontsize{8.83}{10.60}\selectfont \textcolor[rgb]{0, 0.50196, 0}{$\lambda$}}
  		\end{picture}%
  		\caption{\label{fig:fundamental-domain}%
  			A fundamental domain for a cyclic action on $S^3$.}
  	\end{figure}
  	
  	There are some typical elements in $\pi_1(S^3/\phi')$:
  	\begin{itemize}
  		\item $\sigma_1$: represented by the image of the oriented arc
  		\begin{displaymath}
  		\{(\frac{\sqrt{2}}{2},\frac{\sqrt{2}}{2}e^{i\theta}):-\frac{\pi}{p}\leq\theta\leq \frac{\pi}{p}\}.
  		\end{displaymath}
  		$\sigma_1$ has order $p$ and $\psi'(\sigma_1)$ corresponds to the automorphsim $(\phi')^{\psi'(\sigma_1)}$ of $S^3$ defined by
  		\begin{displaymath}
  		(w_1,w_2)\mapsto (w_1,w_2e^{2\pi i\frac{1}{p}}).
  		\end{displaymath} 
  		So $\psi'(\sigma_1)=\alpha_1 ql$, where $\alpha_1$ is a solution of the congruence equation
  		\begin{displaymath}
  		\alpha_1 q m_0\equiv 1 \,({\rm mod}\,p).
  		\end{displaymath}
  		
  		\item $\sigma_2$: represented by the image of the oriented arc
  		\begin{displaymath}
  		\{(\frac{\sqrt{2}}{2}e^{i\theta},\frac{\sqrt{2}}{2}):0\leq\theta\leq \frac{2\pi}{q}\}.
  		\end{displaymath}
  		$\sigma_2$ has order $q$ and $\psi'(\sigma_2)$ corresponds to the map defined by
  		\begin{displaymath}
  		(w_1,w_2)\mapsto (w_1e^{2\pi i\frac{1}{q}},w_2).
  		\end{displaymath} 
  		So $\psi'(\sigma_2)=\alpha_2 pl$, where $\alpha_2$ is a solution of the congruence equation
  		\begin{displaymath}
  		\alpha_2 p \equiv 1 \,({\rm mod}\,q).
  		\end{displaymath}
  		
  		\item $\lambda$: represented by the image of the oriented curve
  		\begin{displaymath}
  		\{(\frac{\sqrt{2}}{2}e^{i\theta/q},\frac{\sqrt{2}}{2}e^{im_0\theta/p}):0\leq\theta\leq \frac{2\pi}{l}\}.
  		\end{displaymath}
  		$\lambda$ has order $n$ and $\psi'(\lambda)$ corresponds to the map defined by
  		\begin{displaymath}
  		(w_1,w_2)\mapsto (w_1e^{2\pi i\frac{1}{ql}},w_2e^{2\pi i\frac{m_0}{pl}}).
  		\end{displaymath}
  		So $\psi'(\lambda)=1$. Note that in $|S^3/\phi'|\cong L(l,l/2+1)$, $\lambda$ is homotopic to a core curve.
  	\end{itemize}
  
  By choosing a canonical generator system 
  \begin{displaymath}
  \mathcal{G}=\{\delta_1,\delta_2,\xi_1,\xi_2,\cdots,\xi_s\}
  \end{displaymath}
  of $\pi_1(\Sigma'/f')$, we may assume $i_*(\xi_i)=(-1)^i\sigma_1(1\leq i\leq t)$ and $i_*(\xi_j)=(-1)^{j-t}\sigma_2(t<j\leq s)$. 
  Then we have 
  \begin{displaymath}
  \begin{split}
  \psi'(\xi_i) &=(-1)^i \alpha ql,\,i=1,2,\cdots,t;\\
  \psi'(\xi_j) &=(-1)^{j-t} \alpha pl,\,j=t+1,t+2,\cdots,s,
  \end{split} 
  \end{displaymath}
  where $\alpha$ is a solution of the congruence equations
  \begin{displaymath}
  \left\{
  \begin{split}
  \alpha q m_0 &\equiv 1\,({\rm mod}\,p),\\
  \alpha p &\equiv 1\,({\rm mod}\,q).
  \end{split} 
  \right.
  \end{displaymath}
  So the isotropy invariant of $f'$ is
  \begin{displaymath}
  (\underbrace{\pm\alpha ql,\cdots,\pm\alpha ql}_{t},\underbrace{\pm\alpha pl,\cdots,\pm\alpha pl}_{s-t}).
  \end{displaymath}
  
  According to Lemma \ref{lem:unique-embedding}, in $|S^3/\phi'|\cong L(l,l/2+1)\cong L(l,l/2-1)$,  $i_*(\delta_1),i_*(\delta_2)$ are homotopic to $\pm\lambda,\pm(l/2-1)\lambda$ respectively. That means
  \begin{displaymath}
  \psi'(\delta_1) \equiv \pm 1,\,
  \psi'(\delta_2) \equiv \pm (l/2-1)\,({\rm mod}\,l).
  \end{displaymath}
  So we have 
  \begin{displaymath}
  h_2(f')=\{\pm 1,\pm(l/2-1)\}\subset\mathbb{Z}_{l/2}.
  \end{displaymath}
  
  If $2\notin\{p,q\}$, $h_1(f')$ is defined. To compute it, as $\psi'(\xi_i)$'s are fixed, we only need to determine the value of $\psi'(\delta_1\delta_2)$. It satisfies 
  \begin{displaymath}
  2(\psi'(\delta_1\delta_2))-\alpha ql-\alpha pl=0\in\mathbb{Z}_{pql}.
  \end{displaymath}
  The equation has two solutions in $\mathbb{Z}_{pql}$: $\alpha(p+q)l/2$ and $\alpha(p+q)l/2+pql/2$. We can exclude one of them with one more branch surgery as follows.
  
  Let $\tilde{\phi}:S^3\to S^3$ be the map defined by
  \begin{displaymath}
  \tilde{\phi}(w_1,w_2)=(w_1e^{2\pi i\frac{1}{2ql}},w_2e^{2\pi i\frac{m_0}{pl}}).
  \end{displaymath}
  Then $(\tilde{\phi})^{pql}$ is a $\pi$-rotation around $Z$ and induces a $2$-fold branched covering map
  \begin{displaymath}
  \tilde{\omega}: S^3 \to |S^3/(\tilde{\phi})^{pql}|\cong S^3.
  \end{displaymath}
  Let $\tilde{\Sigma}$ be the surface $\tilde{\omega}^{-1}(\Sigma')$. Denote the epimorphism for $\tilde{\phi}$ by $\tilde{\psi}$. 
  
  If $q\neq 1$, 
  $\tilde{\Sigma}/\tilde{\phi}$ has the same singular points with $\Sigma'/\phi'$, though the index-$q$ singular points now have index $2q$.
  So $\mathcal{G}$ provides a canonical generator system for $\pi_1(\tilde{\Sigma}/\tilde{\phi})$. With an abuse of notations, similarly we can assume  
  \begin{displaymath}
  \begin{split}
  \tilde{\psi}(\xi_i) &=(-1)^i \tilde{\alpha}\cdot 2ql,\,i=1,2,\cdots,t;\\
  \tilde{\psi}(\xi_j) &=(-1)^{j-t} \tilde{\alpha} pl,\,j=t+1,t+2,\cdots,s,
  \end{split} 
  \end{displaymath}
  where $\tilde{\alpha}$ is a solution of the congruence equations
  \begin{displaymath}
  \left\{
  \begin{split}
  \tilde{\alpha}\cdot 2q m_0 &\equiv 1\,({\rm mod}\,p),\\
  \tilde{\alpha} p &\equiv 1\,({\rm mod}\,2q).
  \end{split} 
  \right.
  \end{displaymath}
  Thus
  \begin{displaymath}
  \tilde{\psi}(\delta_1\delta_2)=\tilde{\alpha}(p+2q)l/2\text{ or }\tilde{\alpha}(p+2q)l/2+pql\in\mathbb{Z}_{2pql}.
  \end{displaymath}
  By Lemma \ref{lem:modulo},
  \begin{displaymath}
  \psi'(\delta_1\delta_2)\equiv\tilde{\psi}(\delta_1\delta_2)\equiv\tilde{\alpha}(p+2q)l/2\,({\rm mod}\,pql).
  \end{displaymath}
  $\tilde{\alpha}$ is unique modulo $2pq$, so $\psi'(\delta_1\delta_2)\in\mathbb{Z}_n$ is determined.
  
  If $q=1$, then $s=t$. 
  The regular points on $(\Sigma'\cap S^1)/\phi'$ now become index-two singular points in $\tilde{\Sigma}/\tilde{\phi}$.
  The number of them is odd, according to Lemma \ref{lem:intersections}. So there are  $\tilde{\xi}_{t+1},\cdots,\tilde{\xi}_{\tilde{s}}$ with $\tilde{s}>s=t$ and $\tilde{s}-t$ odd, such that
  \begin{displaymath}
  \mathcal{G}\cup\{\tilde{\xi}_{t+1},\cdots,\tilde{\xi}_{\tilde{s}}\}
  \end{displaymath}
  is a canonical generator system for $\pi_1(\tilde{\Sigma}/\tilde{\phi})$. Similarly we assume
  \begin{displaymath}
  \begin{split}
  \tilde{\psi}(\xi_i) &=(-1)^i \tilde{\alpha}\cdot 2ql=(-1)^i\tilde{\alpha}\cdot 2l\in\mathbb{Z}_{2pql},\,i=1,2,\cdots,t;\\
  \tilde{\psi}(\tilde{\xi}_j) &=(-1)^{j-t} \tilde{\alpha} pl=pl\in\mathbb{Z}_{2pql},\,j=t+1,t+2,\cdots,\tilde{s},
  \end{split} 
  \end{displaymath}
  Thus we still have
  \begin{displaymath}
  \tilde{\psi}(\delta_1\delta_2)=\tilde{\alpha}(p+2q)l/2\text{ or }\tilde{\alpha}(p+2q)l/2+pql\in\mathbb{Z}_{2pql},
  \end{displaymath}
  \begin{displaymath}
  \psi'(\delta_1\delta_2)\equiv\tilde{\psi}(\delta_1\delta_2)\equiv\tilde{\alpha}(p+2q)l/2\,({\rm mod}\,pql).
  \end{displaymath}
  
  So in either case, $\psi'(\delta_1\delta_2)\equiv\tilde{\alpha}(p+2q)l/2\,({\rm mod}\,pql)$, and $h_1(f')$ can be computed by definition.
  
  Now take the smallest positive integer $k$ that satisfies
  \begin{displaymath}
  \left\{
  \begin{array}{l}
  k \equiv q m_0 \,({\rm mod}\,p),\\
  k \equiv p \,({\rm mod}\,2q),\\
  {\rm gcd}(k,pql)=1.
  \end{array} 
  \right.
  \end{displaymath}
  With Lemma \ref{lem:number-theory} (1) we see such $k$ exists.
  We can simply verify $k\alpha\equiv 1\,({\rm mod}\,p),\,k\alpha\equiv 1\,({\rm mod}\,q)$ and $\,k\tilde{\alpha}\equiv (p+1)/2\,({\rm mod}\,p),\,k\tilde{\alpha}\equiv 1\,({\rm mod}\,2q)$.
  Take $k_0$ such that $kk_0\equiv 1\,({\rm mod}\,pql)$. Let $f_0=(f')^{k_0}$ and denote the corresponding epimorphism by $\psi_0$. By Lemma \ref{lem:power}, 
  \begin{displaymath}
  \begin{split}
  \psi_0(\xi_i) &=k\psi'(\xi_i)=(-1)^ik\alpha ql=(-1)^iql\in\mathbb{Z}_{pql},i=1,2,\cdots,t;\\
  \psi_0(\xi_j) &=k\psi'(\xi_j)=(-1)^{j-t}k\alpha pl=(-1)^{j-t}pl\in\mathbb{Z}_{pql},j=t+1,t+2,\cdots,s.
  \end{split}  
  \end{displaymath}
  So $f_0$ has isotropy invariant 
  \begin{displaymath}
  (\underbrace{\pm ql,\cdots,\pm ql}_{t},\underbrace{\pm pl,\cdots,\pm pl}_{s-t}),
  \end{displaymath}
  and 
  \begin{displaymath}
  h_2(f_0)=\{\pm k,\pm(l/2-k)\}\subset\mathbb{Z}_{l/2}.
  \end{displaymath}
  If $h_1(f_0)$ is defined, i.e., $2\notin\{p,q\}$, then we can compute it according to the definition.
  For example, when $p>2,q>2$,
  as $ql,pl\in\{0,2,\cdots,pql/2-2\}$, we have
  \begin{displaymath}
  \begin{split}
  \chi_k &=
  \begin{cases}
  0, & \textrm{if }\psi_0(\xi_k)\in\{2,4,\cdots,\frac{pql}{2}-2\}\subset\mathbb{Z}_{pql};\\
  1, & \textrm{if }\psi_0(\xi_k)\in\{\frac{pql}{2}+2,\frac{pql}{2}+4,\cdots,pql-2\}\subset\mathbb{Z}_{pql}
  \end{cases}\\
  &=
  \begin{cases}
  1, & \textrm{if } 1\leq k\leq t\text{ and }k\text{ is odd};\\
  0, & \textrm{if } 1\leq k\leq t\text{ and }k\text{ is even};\\
  1, & \textrm{if } t<k\leq s\text{ and }k-t\text{ is odd};\\
  0, & \textrm{if } t<k\leq s\text{ and }k-t\text{ is even},
  \end{cases}
  \end{split}
  \end{displaymath}
  \begin{displaymath}
  \begin{split}
  h_1(f_0) &=\psi_0(\delta_1\delta_2)+\sum_{k=1}^{s}\chi_k\psi_0(\xi_k)\\ &=k\tilde{\alpha}(p+2q)l/2+\frac{t+1}{2}\cdot(-ql)+\frac{s-t+1}{2}\cdot(-pl)\\
  &=pl/2+\frac{p+1}{2}ql-(t+1)ql/2-(s-t+1)pl/2\\
  &=pql/2-tql/2-(s-t)pl/2\in\mathbb{Z}_{pql}.
  \end{split}
  \end{displaymath}
  With similar check for other cases, we finally conclude that
  \begin{displaymath}
  h_1(f_0)=pql/2-\max\{t,1\}ql/2-\max\{s-t,1\}pl/2\in\mathbb{Z}_{pql}.
  \end{displaymath}
\end{example}

  \begin{prop}\label{prop:conjugacy-+}
  	(1) Suppose $f$ satisfies Theorem \ref{thm:-+} (1), then the conjugacy class of $\langle f\rangle$ is uniquely determined by $n,h,s$. 
  	
  	(2) Suppose $f$ satisfies Theorem \ref{thm:-+} (2), then the conjugacy class of $\langle f\rangle$ is uniquely determined by $n,h,s,l$. 
  	
  	(3) Suppose $f$ satisfies Theorem \ref{thm:-+} (3), then the conjugacy class of $\langle f\rangle$ is uniquely determined by $n,h,s,t,p,q$.
  \end{prop}
  
  \begin{proof}
  	Also, $g$ can be figured out from the Riemann-Hurwitz formula, and we may assume $\alpha=1\in\mathbb{Z}_n$. Then (1) follows Theorem \ref{thm:classification}. 
  	For (2), $h_1$ does not exist, and in the case $h=2$ we must have $h_2(f)=\{\pm1,\pm 1\}\subset\mathbb{Z}_2$ for ${\rm gcd}(2\alpha,2l\alpha,n)=2$. Therefore (2) is also verified, and we turn to (3). 
  	
  	The following cases are straightforward.
  	\begin{itemize}
  		\item $n/2$ is odd. Then $l/2$ is odd and $h\neq 2$. So $f$ has no invariants $h_1,h_2$.
  		\item $n/2$ is even, $h\neq 2$ and $2\in\{p,q\}$. Without loss of generality, suppose $q=2$, then by the assumptions in the theorem we have $t\neq s$, and there exists $\pm pl=n/2$ in the isotropy invariant. So $f$ has no invariants $h_1,h_2$.
  		\item $h=2$. 
  	\end{itemize}
  	There is only one remaining case where $n/2$ is even, $h\neq 2$ and $2\notin\{p,q\}$. We only need to focus on the normalization of $h_1(f)$ for that case.  
  	
  	Fix a canonical generator system 
  	\begin{displaymath}
  	\mathcal{G}=\{\delta_1,\cdots,\delta_h,\xi_1,\cdots,\xi_s\}
  	\end{displaymath}
  	such that $\psi(\xi_i)=\alpha ql\,(1\leq i\leq t)$, $\psi(\xi_j)=\alpha pl\,(t+1\leq j\leq s)$.
  	The relation 
  	\begin{displaymath}
  	\prod_{i=1}^{h}\delta_i^2\prod_{k=1}^{s}\xi_k=1\in\pi_1(\Sigma_g/f)
  	\end{displaymath}
  	implies that $x=\sum\limits_{i=1}^{h}\psi(\delta_i)$ is a solution for the following equation:
  	\begin{displaymath}
  	2x+t\alpha ql+(s-t)\alpha pl\equiv 0 \,({\rm mod}\,n).
  	\end{displaymath}
  	There are two solutions:
  	\begin{displaymath}
  	x_1=-(t\alpha ql+(s-t)\alpha pl)/2;\,x_2=x_1+n/2.
  	\end{displaymath}
  	So $\sum\limits_{i=1}^{h}\psi(\delta_i)=x_1$ or $\sum\limits_{i=1}^{h}\psi(\delta_i)=x_1+n/2$.
  	
  	If $pq$ is odd, $n/2\equiv l/2\,({\rm mod}\,l)$. 
  	By definition,
  	\begin{displaymath}
  	h_1(f)=\sum_{i=1}^{h}\psi(\delta_i)+t\chi\alpha ql +(s-t)\chi'\alpha pl\in\mathbb{Z}_n,
  	\end{displaymath}
  	where 
  	\begin{displaymath}
  	\begin{split}
  	\chi &=
  	\begin{cases}
  	0, & \textrm{if }\alpha ql\in\{0,2,4,\cdots,\frac{n}{2}-2\}\subset\mathbb{Z}_n;\\
  	1, & \textrm{if }\alpha ql\in\{\frac{n}{2}+2,\frac{n}{2}+4,\cdots,n-2\}\subset\mathbb{Z}_n;
  	\end{cases}\\
  	\chi' &=
  	\begin{cases}
  	0, & \textrm{if }\alpha pl\in\{0,2,4,\cdots,\frac{n}{2}-2\}\subset\mathbb{Z}_n;\\
  	1, & \textrm{if }\alpha pl\in\{\frac{n}{2}+2,\frac{n}{2}+4,\cdots,n-2\}\subset\mathbb{Z}_n.
  	\end{cases}
  	\end{split}
  	\end{displaymath}
  	Note that $n/2\equiv pql/2\equiv l/2\,({\rm mod}\,l)$. With the condition $h_1(f)\equiv l/2\,({\rm mod}\,l)$, we see that one of $x_1,x_1+n/2$ is impossible for $\sum\limits_{i=1}^{h}\psi(\delta_i)$. Then $h_1(f)$, and therefore $\langle f\rangle$, are determined. 
  	
  	If $pq$ is even, it suffices to show that there exists a positive integer $m$ such that $\langle f^m\rangle=\langle f\rangle$, $h_1(f^m)=h_1(f)+n/2$ and $f^m,f$ have the same isotropy invariant. 
  	In fact, if it holds, then we can always normalize $h_1(f)$ to $x_1$ and thus $\langle f\rangle$ is determined up to conjugacy.
  	By Lemma \ref{lem:number-theory} (1), we choose a positive integer $d$ such that ${\rm gcd}((pq+1)+(2pq)d,n)=1$. Choose $m\in\mathbb{Z}$ such that $m\cdot((pq+1)+(2pq)d)\equiv 1\,({\rm mod}\,n)$. Let $\psi'$
  	be the corresponding epimorphism for $f^m$. By Lemma \ref{lem:power}, for each $\gamma\in\pi_1(\Sigma/f)$, we have $\psi'(\gamma)=((pq+1)+(2pq)d)\psi(\gamma)$. 
  	Without loss of generality, assume $p$ is odd and $q$ is even, hence $s-t$ is odd. If 
  	\begin{displaymath}
  	\sum_{i=1}^{h}\psi(\delta_i)=x_1+n/2=(pq-t\alpha q-(s-t)\alpha p)l/2,
  	\end{displaymath}
  	we have
  	\begin{displaymath}
  	\begin{split}
  	\sum_{i=1}^{h}\psi'(\delta_i)&=((pq+1)+(2pq)d)\sum_{i=1}^{h}\psi(\delta_i)\\
  	&=(pq+1)\sum_{i=1}^{h}\psi(\delta_i)+0\\
  	&=\frac{pq}{2}(pql)-t\alpha q\frac{pql}{2}-(s-t)\alpha(pql)\frac{p}{2}+\sum_{i=1}^{h}\psi(\delta_i)\\
  	&=0+0+\frac{n}{2}+\sum_{i=1}^{h}\psi(\delta_i)\in\mathbb{Z}_n;
  	\end{split} 
  	\end{displaymath}
  	if $\sum\limits_{i=1}^{h}\psi(\delta_i)=x_1$, a similar check also shows 
  	\begin{displaymath}
  	\sum\limits_{i=1}^{h}\psi'(\delta_i)=\sum\limits_{i=1}^{h}\psi(\delta_i)+n/2.
  	\end{displaymath}
  	Moreover, for $i=1,2,\cdots,s$, as $\psi(\xi_i)=\alpha ql\text{ or }\alpha pl$, we have
  	\begin{displaymath}
  	\psi'(\xi_i)=((pq+1)+(2pq)d)\psi(\xi_i)=\psi(\xi_i)\in\mathbb{Z}_n.
  	\end{displaymath} 
  	So $f^m,f$ have the same isotropy invariant and $h_1(f^m)=h_1(f)+n/2$.
  \end{proof}

  Similarly, all extendable maps in type $(-,+)$ can be constructed from Examples \ref{ex:-+_1}, \ref{ex:-+_2}, \ref{ex:-+_3} and \ref{ex:-+_4} with the surgeries, see Table \ref{table:-+}.
  
  \begin{table}[htbp]
  	\caption{Extendable periodic maps in type $(-,+)$.}
  	\label{table:-+}
  	\centering
  	\begin{tabularx}{340pt}{l|X}
  		\toprule
  		\multicolumn{1}{c}{Case (in Theorem \ref{thm:-+})} & \multicolumn{1}{c}{Initial example}\\
  		\midrule
  		(1) & Example \ref{ex:-+_1} with $h=0,s=1$ \\
  		\midrule
  		(2) with $h$ odd & Example \ref{ex:-+_3} (1) with 
  		$n=2l,h=1,s=0$\\
  		\midrule
  		(2) with $h$ even & Example \ref{ex:-+_4} (1) with
  		$n=2l,h=2,s=0$\\
  		\midrule
  		(3) with $h=1$ & Example \ref{ex:-+_2} with
  		$h=1,s=t=0,p=q=1$ \\
  		\midrule
  		(3) with $h$ even & Example \ref{ex:-+_3} (2) with $h=2,s=t=0,p=q=1$\\
  		\midrule
  		(3) with $h$ odd and $h\geq 3$ & Example \ref{ex:-+_4} (2) with $h=3,s=t=0,p=q=1$\\
  		\bottomrule
  	\end{tabularx}
  \end{table}

\section{Necessary conditions for extendability}\label{sect:necessity}
  A periodic automorphism of $S^2$ is conjugate to a rational rotation, maybe composed with a reflection, hence must be extendable in each type with respect to the standard embedding $S^2\hookrightarrow S^3$. So in this section, we assume the genus $g$ is no less than $1$. 
  
  Suppose a periodic map $f\in {\rm Aut}(\Sigma_g)$ of order $n$ is extendable over $S^3$, with respect to an embedding $e:\Sigma_g\hookrightarrow S^3$ and an automorphism $\phi\in {\rm Aut}(S^3)$. For convenience, we identify $\Sigma_g$ with its image $e(\Sigma_g)$. Denote the quotient map by $\rho:(S^3,\Sigma_g)\to(S^3/\phi,\Sigma_g/f)$.
  As we work in the smooth category, $\phi$ can be assumed to be a torsion in the orthogonal group $O(4)$. By the assumption $g\geq1$ we see that $\phi^m=\text{id}$ if and only if $f^m=\text{id}$, so $\phi$ is also of order $n$.
  
  Maps extendable over $S^3$ in type $(+,+)$ are classified  in \cite{NWW}. So we only verify the necessity of the listed conditions in Theorems \ref{thm:--}, \ref{thm:+-} and \ref{thm:-+}.
  
  We denote the fixed point set of an automorphism by ${\rm fix}(\cdot)$. And for a group action $G\curvearrowright X$, we use ${\rm Fix}(G,X)$ to denote the set 
  \begin{displaymath}
  \{x\in X:\exists\,\gamma\in G\backslash\{1_G\}, \textrm{ s.t. } \gamma(x)=x\}.
  \end{displaymath}
  For a periodic map $f\in{\rm Aut}(\Sigma_g)$, if ${\rm Fix}(\langle f \rangle,\Sigma_g)$ has dimension 1, then $f$ must be orientation-reversing; if ${\rm fix}(f)$ has dimension 1, then $f$ must be an orientation-reversing involution; and if there exists an isolated point in ${\rm fix}(f)$, then $f$ must be orientation-preserving. 
  
\subsection{Type $(-,-)$}\label{subsect:--}
  \begin{prop}\label{prop:--equivalence}
  	Suppose $f\in {\rm Aut}(\Sigma_g)$ is orientation-reversing. Then $f$ is extendable over $S^3$ in type $(-,-)$ if and only if it is extendable over $\mathbb{R}^3$.
  \end{prop}
  
  \begin{proof}
  	Suppose $f$ extends to some orientation-reversing automorphism $\phi$ of $S^3$, with respect to some embedding $\Sigma_g\hookrightarrow S^3$. $\phi\in O(4)$ has the standard form up to similarity:
  	\begin{displaymath}
  	\left(
  	\begin{array}{cccc}
  	\cos\frac{2p\pi}{n} & -\sin\frac{2p\pi}{n} & & \\
  	\sin\frac{2p\pi}{n} & \cos\frac{2p\pi}{n} & & \\
  	& & 1 &  \\
  	& & & -1
  	\end{array}
  	\right),
  	\gcd(n,p)=1\,\text{or}\,2.
  	\end{displaymath}
  	Identify $S^3$ with $\mathbb{R}^3\cup\{\infty\}$, then $\phi$ corresponds to the matrix 
  	\begin{displaymath}
  	\left(
  	\begin{array}{ccc}
  	\cos\frac{2p\pi}{n} & -\sin\frac{2p\pi}{n} & \\
  	\sin\frac{2p\pi}{n} & \cos\frac{2p\pi}{n} & \\
  	& & -1  
  	\end{array}
  	\right),
  	\end{displaymath}
  	which fixes two points $0,\infty$. If $\infty\notin \Sigma_g$, $f$ is naturally extendable over $\mathbb{R}^3$. Otherwise $\infty\in{\rm fix}(f)$, then $f$ has a 1-dimensional fixed point set, for an orientation-reversing map has no isolated fixed point. It implies $f$ is an involution and $\phi={\rm diag}(1,1,-1)$. So $|\Sigma_g/f|$ is homeomorphic to the closure of a fundamental domain, which is an orientable surface with non-empty boundary. According to Proposition 1 in \cite{C2}, $f$ is extendable over $\mathbb{R}^3$ as well.
  \end{proof}
  
  The extendability over $\mathbb{R}^3$ of an orientation-reversing map has been discussed by Costa \cite{C2,C3}. So Theorem \ref{thm:--} is essentially the same with his conclusions. We omit the details in the calculation here.
  
\subsection{Type $(+,-)$}\label{subsect:+-}
  $\phi$ has the same matrix as in the proof of Proposition \ref{prop:--equivalence} with $n$ even. For the orientation reason, $\phi$ exchanges the two components of $S^3-\Sigma_g$, so the fixed points $0,\infty$ must be on $\Sigma_g$. Note that ${\rm Fix}(\langle f \rangle,\Sigma_g)$ is discrete and $|\Sigma_g/f|$ is a closed orientable surface. 
  If $\gcd(n,p)=2$, $n/2$ must be odd as $\phi$ is of order $n$. Then ${\rm fix}(\phi^\frac{n}{2})$ is the sphere $\{\infty\}\cup xy$-plane. This implies ${\rm fix}(f^\frac{n}{2})=\Sigma_g\cap{\rm fix}(\phi^\frac{n}{2})$ has dimension 1 and $f^{\frac{n}{2}}$ is orientation-reversing, a contradiction.
  So $\gcd(n,p)=1$. Without loss of generality, we assume $p=1$, for otherwise we can consider $f^k$ instead of $f$, where $k\in\mathbb{Z}$ satisfies $kp\equiv 1\,({\rm mod}\,n)$. Then $\phi$ is the composition of a $\frac{2\pi}{n}$-rotation around the $z$-axis and the reflection across the $xy$-plane. 
  
  If $n=2$, ${\rm Fix}(\langle\phi\rangle,S^3)=\{0,\infty\}$. Hence $\Sigma_g/f$ has only two singular points, and their corresponding elements $\xi_1,\xi_2\in\pi_1(\Sigma_g/f)$ are sent to the generator of $\mathbb{Z}_2$ by the epimorphism $\psi:\pi_1(\Sigma_g/f)\to\mathbb{Z}_2$.
  
  If $n>2$, $\Sigma_g$ intersects the circle $\{\infty\}\cup z$-axis at $2s-2$ points: $0,\infty,(0,0,\pm z_i)(1\leq i\leq s-2)$, where $0<z_1<z_2<\cdots<z_{s-2}$. Their images in $\Sigma_g/f$ are the singular points. Let $0,\infty,(0,0,z_i)$ correspond to $\xi_1,\xi_2,\xi_{i+2}\in\pi_1(\Sigma_g/f)$ respectively. Then $\psi(\xi_1)=1$, and $\psi(\xi_i)=(-1)^{i}\times 2$ for $3\leq i\leq s$. $\psi(\xi_2)$ is determined by
  \begin{displaymath}
  0=\psi(\prod_{i=1}^{h}[\alpha_i,\beta_i]\prod_{k=1}^{s}\xi_k)=\sum_{k=1}^{s}\psi(\xi_k),
  \end{displaymath}
  so equals $-1$ if $s$ is even and otherwise $1$. 
  
\subsection{Type $(-,+)$}\label{subsect:-+}
  As $\phi\in O(4)$ is orientation-preserving, it has the standard form 
  \begin{displaymath}
  \left(
  \begin{array}{cccc}
  \cos\frac{2m_1\pi}{n} & -\sin\frac{2m_1\pi}{n} & & \\
  \sin\frac{2m_1\pi}{n} & \cos\frac{2m_1\pi}{n} & & \\
  & & \cos\frac{2m_2\pi}{n} & -\sin\frac{2m_2\pi}{n} \\
  & & \sin\frac{2m_2\pi}{n} & \cos\frac{2m_2\pi}{n}
  \end{array}
  \right),
  \gcd(n,m_1,m_2)=1,
  \end{displaymath}
  up to similarity. 
  Let $p={\rm gcd}(n,m_1),q={\rm gcd}(n,m_2)$, then $p,q$ are coprime as ${\rm gcd}(n,m_1,m_2)=1$. Without loss of generality, we set $p$ to be odd.
  Moreover, we assume $m_1=p$, for otherwise we can consider $\phi^k,f^k$ instead of $\phi,f$, where $k$ satisfies ${\rm gcd}(k,n)=1$ and $km_1\equiv p\,({\rm mod}\,n)$ (Lemma \ref{lem:number-theory} (2)). Assume $n=pql$ and $m_2=mq$.
  Use the model 
  \begin{displaymath}
  S^3=\{(w_1,w_2)\in\mathbb{C}^2:|w_1|^2+|w_2|^2=1\},
  \end{displaymath}
  then $\phi$ is the map 
  \begin{displaymath}
  (w_1,w_2)\mapsto (w_1e^{2\pi i\frac{1}{ql}},w_2e^{2\pi i\frac{m}{pl}}).
  \end{displaymath}
  The $\phi$-action on $S^3$ has a fundamental domain $F$ as in Figure \ref{fig:fundamental-domain}, and $|S^3/\phi|$ is homeomorphic to a lens space $L(l,m)$.  

  If $|\Sigma_g/f|$ is orientable with non-empty boundary, 
  then the boundary comes from the one dimensional part of ${\rm fix}(f^{\frac{n}{2}})\subseteq{\rm fix}(\phi^{\frac{n}{2}})$.
  As
  \begin{displaymath}
  \phi^{\frac{n}{2}}(w_1,w_2)=(w_1e^{2\pi i\frac{p}{2}},w_2e^{2\pi i\frac{mq}{2}})=(w_1e^{2\pi i\frac{1}{2}},w_2e^{2\pi i\frac{mq}{2}}),
  \end{displaymath}
  ${\rm fix}(\phi^{\frac{n}{2}})$ can only be the axis $Z$. So $Z\subset\Sigma_g$ and $\partial |\Sigma_g/f|$ is connected. With the fundamental domain $F$ we see $q=2$, for otherwise $\Sigma_g$ can not be a closed surface. Moreover, by Proposition \ref{prop:core} (1) we have $l=1$. So $\epsilon_1\in\pi_1(\Sigma_g/f)$, the element represented by the oriented boundary $\rho(Z)$, is sent to $2\alpha$ by $\psi$, where $\alpha$ is a generator of $\mathbb{Z}_n$. 
  If $n>2$, the sigular points come from $\Sigma_g\cap S^1$, so $\psi(\xi_i)=\pm 2\alpha$. The signs of $\psi(\xi_1),\cdots,\psi(\xi_s)$ must be alternating, as $|\Sigma_g/f|$ is orientable. From the relation
  \begin{displaymath}
  \psi(\epsilon_1)+\psi(\xi_1)+\cdots+\psi(\xi_s)=0\in\mathbb{Z}_n,
  \end{displaymath}
  we see $s$ is odd and $\psi(\xi_1),\cdots,\psi(\xi_s)$ are $-2\alpha,2\alpha,-2\alpha,2\alpha,\cdots,-2\alpha$.
  
  If $|\Sigma_g/f|$ is non-orientable with non-empty boundary, similarly we have $\rho(Z)=\partial |\Sigma_g/f|$, $q=2$ and $\psi(\epsilon_1)=\pm 2\alpha$, $\langle\alpha\rangle=\mathbb{Z}_n$. Moreover, if $p>1$, $s$ is odd and $\pm\psi(\xi_i)=\pm 2l\alpha$ for each $i=1,2,\cdots,s$.
  
  Finally we suppose $|\Sigma_g/f|$ is a closed (thus non-orientable) surface of genus $h$, embedded in the lens space $L(l,m)$. According to \cite{BW} (see Lemmas \ref{lem:parity}, \ref{lem:closed-surface}), $l$ is even and $l/2,h$ have the same parity. Moreover, if $h=1$, $l$ must be $2$. The circles $\rho(S^1),\rho(Z)\subset S^3/\phi$ have indices $p,q$ respectively. Suppose in $\Sigma_g\cap F$ there are $t$ points on $\rho(S^1)$ and $s-t$ points on $\rho(Z)$. They are the singular points of indices $p,q$ respectively (though they degenerate when $p=1$ or $q=1$).  
  So the isotropy invariant is 
  \begin{displaymath}
  (\underbrace{\pm\beta,\cdots,\pm\beta}_{t},\underbrace{\pm\gamma,\cdots,\pm\gamma}_{s-t}),
  \end{displaymath} 
  where $\beta,\gamma\in\mathbb{Z}_{n}$ are of orders $p,q$ respectively. Moreover, $t,s-t$ should be odd according to Lemma \ref{lem:intersections}.

  When $n/2$ is even and $2\notin\{p,q\}$, $h_1(f)$ is defined. Let 
  \begin{displaymath}
  \phi_l:(w_1,w_2)\mapsto (w_1e^{2\pi i\frac{m}{l}},w_2e^{2\pi i\frac{1}{l}})
  \end{displaymath}
  be the automorphism of $|S^3/\phi^l|\cong S^3$ induced by $\phi$, and $f_l$ be induced on $|\Sigma_g/f^l|$, with the corresponding epimorphism $\psi_l$. 
  As ${\rm Fix}(\langle f_l\rangle,|\Sigma_g/f^l|)\subseteq{\rm Fix}(\langle\phi_l\rangle,S^3)$ and ${\rm Fix}(\langle\phi_l\rangle,S^3)$ is empty, the map 
  \begin{displaymath}
  (S^3,|\Sigma_g/f^l|)\xrightarrow{(\phi_l,f_l)}(L(l,m),|\Sigma_g/f|)
  \end{displaymath}
  is a covering, which induces a commutative diagram: 
  \begin{displaymath}
  	\xymatrix{
  1\ar[r] & \pi_1(|\Sigma_g/f^l|)\ar[r] & \pi_1(|\Sigma_g/f|)\ar[r]^-{\psi_l}\ar[d]_{i_*} & \mathbb{Z}_l\ar[d]^{f_l\leftrightarrow 1\leftrightarrow \phi_l}_{=}\ar[r] & 1\\
  1\ar[r] & \pi_1(S^3)\ar[r] & \pi_1(L(l,m))\ar[r]^-{\cong} & \mathbb{Z}_l\ar[r] & 1.
  }
  \end{displaymath}
  By Proposition \ref{prop:torsion}, $i_*(\prod\limits_{i=1}^{h}\delta_i)$ is non-trivial and represents the order-$2$ element in $H_1(L(l,m);\mathbb{Z})\cong\pi_1(L(l,m))$.
  Hence by Lemma \ref{lem:modulo},
  \begin{displaymath}
  \psi(\prod_{i=1}^{h}\delta_i)\equiv\psi_l(\prod_{i=1}^{h}\delta_i)\equiv l/2\,({\rm mod}\,l).
  \end{displaymath}
  So by definition we have $h_1(f)\equiv l/2\,({\rm mod}\,l)$.
  
  When $h=2$, by Lemma \ref{lem:closed-surface} we have $l\equiv 0\,({\rm mod}\,4)$ and $m\equiv l/2\pm 1\,({\rm mod}\,l)$. 
  Suppose $m\equiv l/2+ 1\,({\rm mod}\,l)$. 
  Take the smallest $m_0\in\mathbb{N}$ such that
  \begin{displaymath}
  \left\{\begin{array}{l}
  m_0\equiv l/2+1 \,({\rm mod}\,l),\\
  {\rm gcd}(m_0,p)=1.
  \end{array}
  \right.
  \end{displaymath}
  Choose $\kappa\in\mathbb{N}$ that satisfies the congruence equations
  \begin{displaymath}
  \left\{
  \begin{array}{clc}
  \kappa\equiv 1 &({\rm mod}\,ql),\\
  \kappa m\equiv m_0 &({\rm mod}\,pl).
  \end{array} 
  \right.
  \end{displaymath}
  With Lemma \ref{lem:number-theory} (3) and a few basic calculations, such $\kappa$ exists and is coprime to $n$. Take $\kappa'$ with $\kappa\kappa'\equiv 1\,({\rm mod}\,n)$. By replacing $\phi,f$ by $\phi^{\kappa'},f^{\kappa'}$, we may assume $m=m_0$. Note that in Example \ref{ex:-+Klein} the embedding of surface has nothing to do with the computation of the invariants. Therefore, we can follow it directly and see that $f$ must have the same isotropy invariant, $h_1,h_2$ as $f'$ in the example. So $\langle f\rangle$ is conjugate to $\langle f_0\rangle$. 
  If $m\equiv l/2-1\,({\rm mod}\,l)$, 
  as the quotient space $|S^3/\phi|\cong L(l,l/2-1)\cong L(l,l/2+1)$, 
  it must be conjugate to the last case, so we finish the proof.
  
  \section{Surfaces in lens spaces}\label{sect:surfaces-in-lens-space}
  In this section, we present some facts about embedded surfaces in lens spaces. 
  
  A lens space $L(l,m)$ (${\rm gcd}(l,m)=1$) can be constructed by gluing two solid tori $V_1=D^2\times S^1,V_2=S^1\times D^2$ with a homeomorphism $\omega:\partial V_2\to \partial V_1$ such that its restriction on the meridian $\{1\}\times \partial D^2$ of $V_2$ is
  \begin{align*}
  \{1\}\times S^1 & \to S^1\times S^1=\partial V_1\\
  (1,e^{2\pi it}) & \mapsto (e^{2\pi itm},e^{2\pi itl}). 
  \end{align*}
  Cutting $V_1$ along the disk $D^2\times\{1\}$, we obtain a cylinder $C=D^2\times I$. $\partial C$ consists of two disks $D^2\times\{0\},D^2\times\{1\}$ and an annulus $\partial D^2\times I=S^1\times I$, denoted by $D_0,D_1,A$ respectively.
  
  \begin{rem}
  	To some degree, $V_1,V_2\subset L(l,m)$ are unique. In fact, for a given lens space, its Heegaard splitting of a fixed genus is unique up to isotopy \cite{BO}.
  \end{rem}
  
  A simple closed curve in a lens space is a \emph{core} if its complement is homeomorphic to a solid torus.
  
  \begin{prop}\label{prop:core}
  	Let $\gamma$ be a core curve of a lens space $L(l,m)$ with ${\rm gcd}(l,m)=1$.
  	
  	(1) If $l>1$, $\gamma$ can not bound an embedded orientable surface in $L(l,m)$.
  	
  	(2) If $l$ is even, $\gamma$ can neither bound an embedded non-orientable surface in $L(l,m)$.
  	
  	(3) If $l$ is odd, $\gamma$ bounds an embedded non-orientable surface in $L(l,m)$. Moreover, if $m=2$, $\gamma$ bounds a M\"obius band; and if $m=4$, $\gamma$ bounds a once-holed Klein bottle.
  \end{prop}
  
  \begin{proof}
  	(1) If $\gamma$ bounds an embedded orientable surface then it must be trivial in $H_1(L(l,m);\mathbb{Z})\cong\pi_1(L(l,m))$, a contradiction. 
  	
  	(2) If $\gamma$ bounds an embedded non-orientable surface then it must be trivial in $H_1(L(l,m);\mathbb{Z}_2)\cong\mathbb{Z}_2$, a contradiction.
  	
  	(3) It suffices to consider the quotient space pair $(|S^3/\phi|,|\Sigma/\phi|)$ in Examples \ref{ex:-+_3} (1) and \ref{ex:-+_4} (1).
  \end{proof}
  
  \begin{lem}[ll.24-28 in p.97 in \cite{BW}] \label{lem:intersections}
  	Suppose a closed surface $\Sigma$ is embedded in a lens space, and $\Sigma$ intersects a core curve transversely. Then $\Sigma$ is non-orientable if and only if the number of their intersection points is odd.
  \end{lem}
  
  \begin{lem}[ll.11-24 in p.88 in \cite{BW}]\label{lem:parity}
  	Suppose a lens space $L(l,m)$ $({\rm gcd}(l,m)=1)$ admits an embedded non-orientable closed surface of genus $h$, then $l$ is even and $l/2$ has the same parity with $h$.
  \end{lem}
  
  \begin{lem}[Theorem 6.1 in \cite{BW}]
  	\label{lem:closed-surface}
  	(1) A lens space $L(l,m)$ admits an embedded projective plane if and only if $L(l,m)\cong L(2,1)$.
  	
  	(2) $L(l,m)$ admits an embedded Klein bottle if and only if $L(l,m)\cong L(4r, 2r\pm 1)$ for some positive integer $r$.
  	
  	(3) A non-orientable closed surface of genus $3$ can be embedded into $L(4r+2,2r-1)$ for each positive integer $r$.
  \end{lem}
  
  According to \cite{BW}, the embeddings can be constructed as follows. For $L(2,1)$, see Figure \ref{fig:projective-plane}.
  \begin{figure}[htbp]
  	\centering
  	\setlength{\unitlength}{1bp}%
  	\begin{picture}(234.12, 84.40)(0,0)
  		\put(0,0){\includegraphics{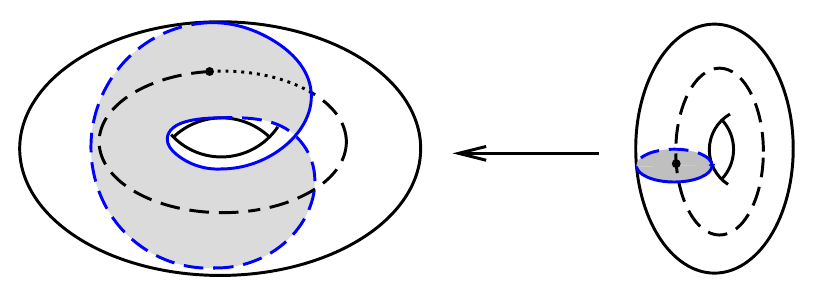}}
  		\put(13.68,8.07){\fontsize{8.83}{10.60}\selectfont $V_1$}
  		\put(179.70,8.15){\fontsize{8.83}{10.60}\selectfont $V_2$}
  		\put(143.56,44.38){\fontsize{8.83}{10.60}\selectfont gluing}
  	\end{picture}%
  	\caption{\label{fig:projective-plane}%
  		A projective plane in L(2,1).}
  \end{figure}
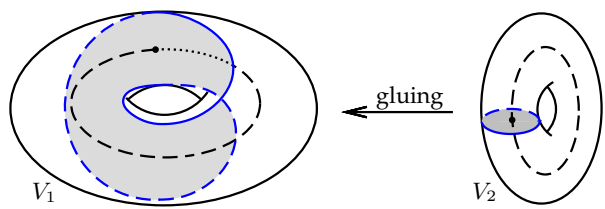
  For $L(4r,2r-1)$, choose the arcs on $\partial C$ as in Figure \ref{fig:Klein-bottle}.
  \begin{figure}[htbp]
  	\centering
  	\setlength{\unitlength}{1bp}%
  	\begin{picture}(360.81, 100.73)(0,0)
  		\put(0,0){\includegraphics{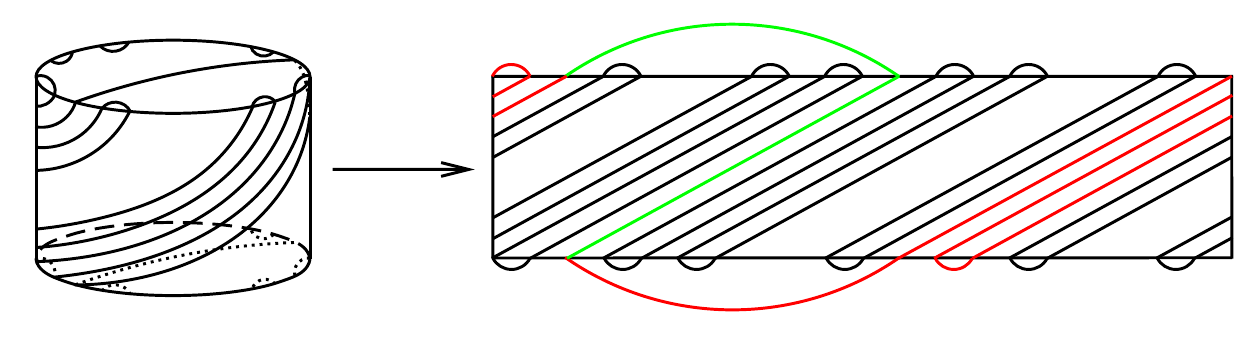}}
  		\put(169.75,62.74){\fontsize{10.69}{12.83}\selectfont $\cdots$}
  		\put(268.33,53.83){\fontsize{10.69}{12.83}\selectfont $\cdots$}
  		\put(333.29,36.51){\fontsize{10.69}{12.83}\selectfont $\cdots$}
  		\put(34.25,49.53){\fontsize{9.96}{11.95}\selectfont $\cdots$}
  		\put(77.48,37.10){\fontsize{8.54}{10.24}\selectfont $\cdots$}
  		\put(5.67,22.54){\fontsize{8.54}{10.24}\selectfont $1$}
  		\put(11.66,13.34){\fontsize{8.54}{10.24}\selectfont $2$}
  		\put(18.65,10.86){\fontsize{8.54}{10.24}\selectfont $3$}
  		\put(26.11,8.75){\fontsize{8.54}{10.24}\selectfont $4$}
  		\put(35.75,7.51){\fontsize{8.54}{10.24}\selectfont $5$}
  		\put(84.41,13.23){\fontsize{8.54}{10.24}\selectfont $2r$}
  		\put(90.71,23.02){\fontsize{8.54}{10.24}\selectfont $2r+1$}
  		\put(89.93,76.58){\fontsize{8.54}{10.24}\selectfont $2r+1$}
  		\put(86.49,84.19){\fontsize{8.54}{10.24}\selectfont $2r+2$}
  		\put(5.87,75.51){\fontsize{8.54}{10.24}\selectfont $1$}
  		\put(8.04,84.62){\fontsize{8.54}{10.24}\selectfont $4r$}
  		\put(96.68,56.54){\fontsize{8.54}{10.24}\selectfont unfolding}
  		\put(249.91,87.29){\fontsize{9.96}{11.95}\selectfont \textcolor[rgb]{0, 1, 0}{$\delta_1$}}
  	\end{picture}%
  	\caption{\label{fig:Klein-bottle}%
  		A Klein bottle in $L(4r,2r-1)$.}
  \end{figure}
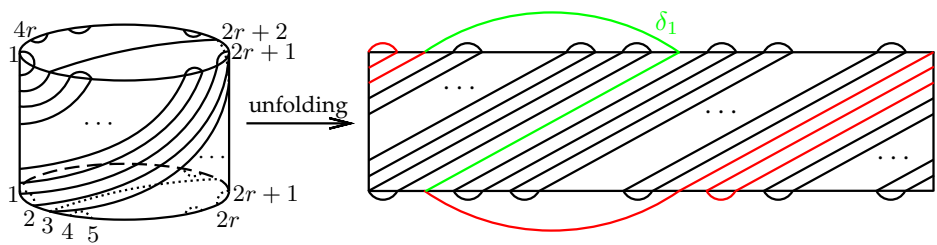
  Their union consists of $2r-1$ closed curves, which bound disjoint disks in $C$. Gluing these disks and $\{1\}\times D^2\subset V_2$ back, we get an embedded closed surface in $L(4r,2r-1)$. From the gluing we see its Euler number is $(2r-1)-(4r)/2+1=0$. As the surface intersects the core circle $S^1\times\{0\}\subset V_2$ at only one point, it must be non-orientable thus a Klein bottle. 
  \begin{figure}[htbp]
  	\centering
  	\setlength{\unitlength}{1bp}%
  	\begin{picture}(354.74, 248.93)(0,0)
  		\put(0,0){\includegraphics{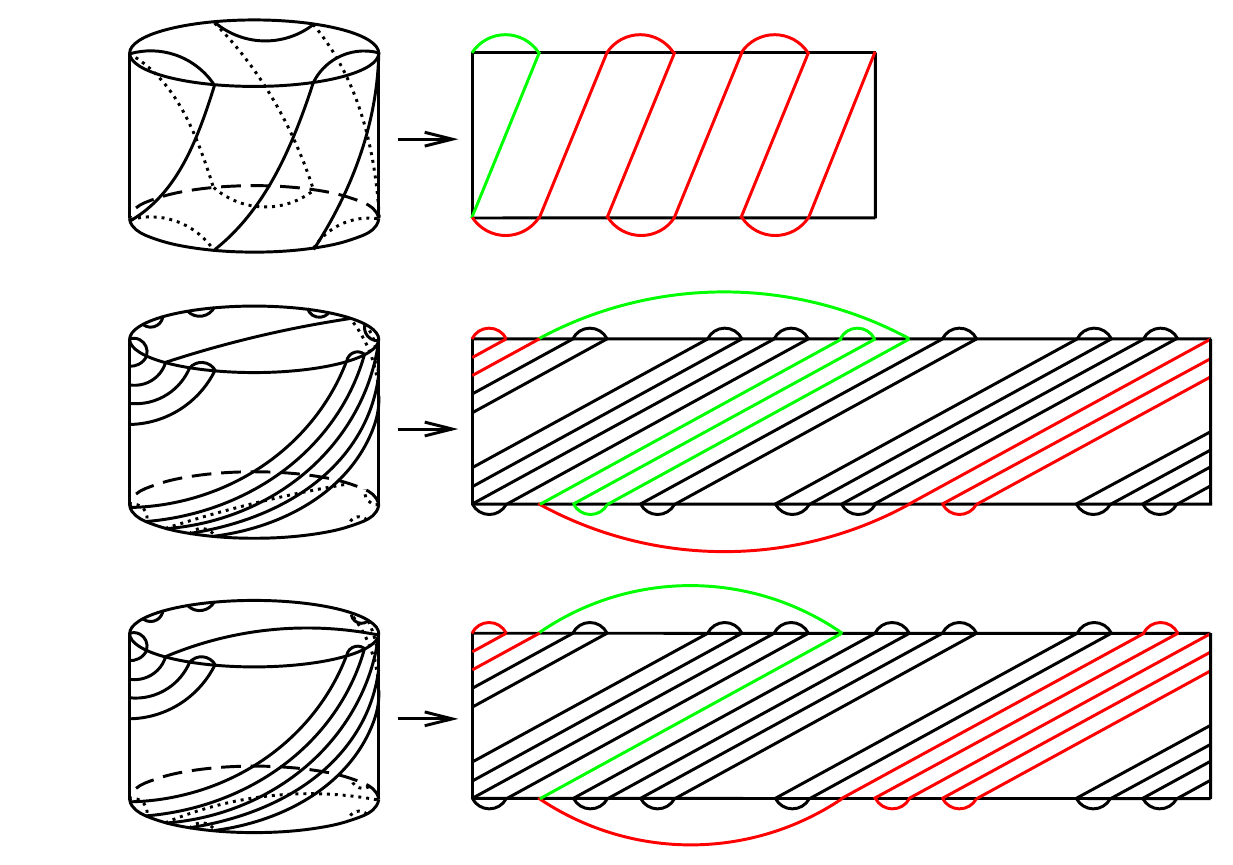}}
  		\put(47.06,146.34){\fontsize{7.76}{9.31}\selectfont $3$}
  		\put(110.14,148.98){\fontsize{7.76}{9.31}\selectfont $2r+2$}
  		\put(101.28,158.05){\fontsize{7.76}{9.31}\selectfont $2r+4$}
  		\put(32.54,149.22){\fontsize{7.76}{9.31}\selectfont $1$}
  		\put(46.24,61.64){\fontsize{7.76}{9.31}\selectfont $3$}
  		\put(110.40,63.72){\fontsize{7.76}{9.31}\selectfont $2r+2$}
  		\put(32.59,64.52){\fontsize{7.76}{9.31}\selectfont $1$}
  		\put(5.67,205.70){\fontsize{9.96}{11.95}\selectfont $r=1$}
  		\put(5.67,128.97){\fontsize{9.96}{11.95}\selectfont $r>1$}
  		\put(5.67,39.35){\fontsize{9.96}{11.95}\selectfont $r$ even}
  		\put(161.35,136.59){\fontsize{9.72}{11.66}\selectfont $\cdots$}
  		\put(250.97,128.49){\fontsize{9.72}{11.66}\selectfont $\cdots$}
  		\put(59.04,124.58){\fontsize{9.72}{11.66}\selectfont $\cdots$}
  		\put(73.21,141.02){\fontsize{9.72}{11.66}\selectfont $\cdots$}
  		\put(68.85,156.84){\fontsize{9.72}{11.66}\selectfont $\cdots$}
  		\put(81.99,93.70){\rotatebox{6.00}{\fontsize{9.72}{11.66}\selectfont \smash{\makebox[0pt][l]{$\cdots$}}}}
  		\put(161.35,51.82){\fontsize{9.72}{11.66}\selectfont $\cdots$}
  		\put(250.97,43.71){\fontsize{9.72}{11.66}\selectfont $\cdots$}
  		\put(310.02,27.97){\fontsize{9.72}{11.66}\selectfont $\cdots$}
  		\put(59.04,39.81){\fontsize{9.72}{11.66}\selectfont $\cdots$}
  		\put(73.21,56.24){\fontsize{9.72}{11.66}\selectfont $\cdots$}
  		\put(68.85,72.07){\fontsize{9.72}{11.66}\selectfont $\cdots$}
  		\put(79.55,8.99){\rotatebox{5.00}{\fontsize{9.72}{11.66}\selectfont \smash{\makebox[0pt][l]{$\cdots$}}}}
  		\put(310.80,112.74){\fontsize{9.72}{11.66}\selectfont $\cdots$}
  		\put(5.67,117.66){\fontsize{9.96}{11.95}\selectfont $r$ odd}
  	\end{picture}%
  	\caption{\label{fig:genus-3-surface}%
  		A genus-$3$ non-orientable surface in $L(4r+2,2r-1)$.}
  \end{figure}
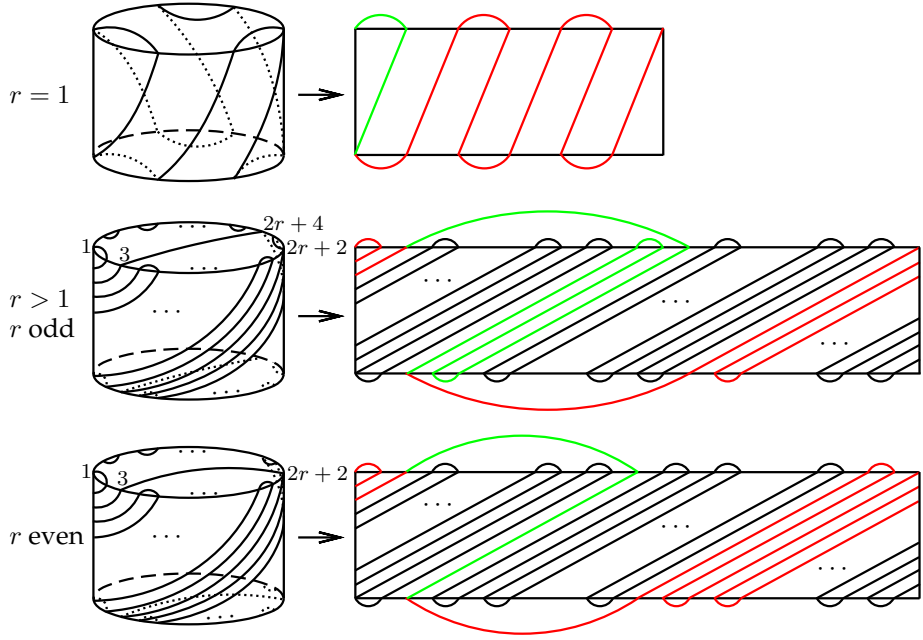
  Similarly, with the closed curves on $\partial C$ shown in Figure \ref{fig:genus-3-surface}, we obtain a non-orientable closed surface of genus 3, embedded in $L(4r+2,2r-1)$.
  
  \begin{rem}\label{rem:surjectivity}
  	Every green curve in the Figures generates the fundamental group of the corresponding lens space. Thus the embeddings above induce surjective homomorphisms on $\pi_1$.
  \end{rem}
  
  \begin{prop}\label{prop:torsion}
  	Suppose $K$ is a non-orientable closed surface of genus $h$, and $i:K\to L(l,m)$ is an embedding. Let $\delta\in H_1(K;\mathbb{Z})$ be the order-$2$ element, then $i_*(\delta)$ is non-trivial in $H_1(L(l,m);\mathbb{Z})$.
  \end{prop}
  
  \begin{proof}
  	We follow the arguments in Section 7 of \cite{BW}. Suppose $K$ intersects the core $S^1\times \{0\}\subset V_2$ transversely, thus $K\cap V_2$ can be assumed as a union of $d$ meridian disks. We may reduce to the case $d=1$ by isotopy as \cite{BW} does. Now assume $K\cap V_2=\{1\}\times D^2$ and $K\cap \partial C$ consists of $l$ parallel arcs on $A$, $l/2$ arcs on $D_0$, $l/2$ arcs on $D_1$, and maybe some closed curves in the interior of $D_0,D_1$. Thus $K\cap C$ consists of some compact surfaces properly embedded in $C$. They must be orientable, otherwise a double of $C$ will give an $S^3$ admitting embedded non-orientable closed surfaces. Now $K$ can be constructed by gluing back these surfaces and the disk $K\cap V_2$.
  	We just glue part of them back as follows, to obtain a compact orientable surface $K_0$, such that $K_0$ can also be constructed by cutting $K$ along a union of curves which represents $\delta$ in homology. 
  	
  	Fix an orientation on $K\cap V_2$. For each arc component $\alpha$ of $K\cap D_1$, it connects two parallel components $\beta_1,\beta_2$ of $K\cap A$. Let $K_{\alpha}$ be the component of $K\cap C$ such that $\partial K_{\alpha}\supset \alpha\cup\beta_1\cup\beta_2$. For any given orientation on $K_{\alpha}$, it induces opposite orientations on the parallel arcs $\beta_1,\beta_2$. Therefore, we can glue part of $\partial(K\cap V_2)$ back to one of $\beta_1,\beta_2$, such that the orientation of $K\cap V_2$ and that of $K_{\alpha}$ concide, thus $(K\cap V_2)\cup K_{\alpha}$ is well oriented. In this way, half of the $l$ parallel arcs $K\cap A$ can be glued onto $\partial (K\cap V_2)$, such that the obtained surface, denoted $K_0$, is still orientable. 
  	$K$ can be obtained by gluing $\partial K_0$, so $\partial K_0$ represents $2\delta$ in $H_1(K;\mathbb{Z})$. Geometrically, $\partial K_0$ consists of $l/2$ parallel arcs on $A$, $l/2$ arcs in $\partial(K\cap V_2)$ and some components of $K_0\cap(D_0\cup D_1)$. Thus $i_*(\delta)=l/2$ in $H_1(L(l,m);\mathbb{Z})\cong\mathbb{Z}_l$. 
  \end{proof}
  
  The following lemma can also be deduced by the analysis in Section 7 in \cite{BW} and Figure \ref{fig:Klein-bottle}. Note that the green curve in Figure \ref{fig:Klein-bottle} represents $\delta_1$. 
  
  \begin{lem}\label{lem:unique-embedding}
  	The embedding of a Klein bottle into $L(4r,2r-1)$ $(r\geq 2)$ is unique up to isotopy. Denote it by $i$ and identify $H_1(L(4r,2r-1);\mathbb{Z})$ with $\mathbb{Z}_{4r}$ by mapping an oriented core to $1$. Let $\langle\delta_1,\delta_2\,|\,\delta_1^2\delta_2^2=1\rangle$ be the fundamental group of the Klein bottle, then
  	\begin{displaymath}
  	\{i_*([\delta_1]),i_*([\delta_2])\}=\{1,2r-1\} \text{ or }\{-1,2r+1\}\subset\mathbb{Z}_{4r}
  	\end{displaymath}
  	in homology.
  \end{lem}
  
\bibliographystyle{amsalpha}

\begin{thebibliography}{10}

\bibitem{BO}
Bonahon, F. and Otal, J.~P., {\em Scindements de Heegaard des espaces lenticulaires}, Annales scientifiques de l'\'Ecole Normale Sup\'erieure 16.3 (1983): 451-466.

\bibitem{BW}
Bredon, G.~E. and Wood, J.~W., {\em Non-orientable surfaces in orientable 3-manifolds}, Inventiones Mathematicae 7.2 (1969): 83-110.

\bibitem{C1}
Costa, A.~F., {\em Classification of the orientation reversing homeomorphisms of finite order of surfaces}, Topology and its Applications 62.2 (1995): 145-162.

\bibitem{C2}
Costa, A.~F., {\em Embeddable anticonformal automorphisms of Riemann surfaces}, Commentarii Mathematici Helvetici 72.2 (1997): 203-215.

\bibitem{C3}
Costa, A.~F., {\em 
Note on the classification of the orientation
reversing homeomorphisms of finite order of
surfaces}, arXiv:2006.03904 [math.GT] (2020).

\bibitem{FK}
Funayoshi, K. and Koda, Y., {\em Extending automorphisms of the genus-2 surface over the 3-sphere}, The Quarterly Journal of Mathematics 71.1 (2020): 175-196.

\bibitem{GWWZ}
Guo, Y., Wang, C., Wang, S.~C. and Zhang, Y.~M., {\em Embedding periodic maps on surfaces into those on $S^3$}, Chinese Annals of Mathematics, Series B 36.2 (2015): 161-180.

\bibitem{NWW}
Ni, Y., Wang, C. and Wang, S.~C., {\em Extending periodic automorphisms of surfaces to 3-manifolds}, arXiv:2003.11773 (2020).

\bibitem{R}
R\"uedy, R.~A., {\em Symmetric embeddings of Riemann surfaces}. Discontinuous Groups and Riemann Surfaces (AM-79), Volume 79: Proceedings of the 1973 Conference at the University of Maryland. (AM-79). Vol. 79. Princeton University Press, 2016.

\bibitem{WWZZ1}
Wang, C., Wang, S.~C., Zhang, Y.~M., and Zimmermann, B., {\em Extending finite group actions on surfaces over $S^3$}, Topology and its Applications 160.16 (2013): 2088-2103.

\bibitem{WWZZ2}
Wang, C., Wang, S.~C., Zhang, Y.~M., and Zimmermann, B., {\em Embedding surfaces into $ S^ 3$ with maximum symmetry}, Groups, Geometry, and Dynamics 9.4 (2015): 1001-1045.

\bibitem{WZ}
Wang, C., and Zhang, Y.~M., {\em Maximum orders of cyclic and abelian extendable actions on surfaces}, Colloquium Mathematicum 154.2 (2018): 183-204.

\bibitem{Y1}
Yokoyama, K., {\em Classification of periodic maps on compact surfaces: I}, Tokyo Journal of Mathematics 6.1 (1983): 75-94.

\bibitem{Y2}
Yokoyama, K., {\em Classification of periodic maps on compact surfaces: II}, Tokyo Journal of Mathematics 7.1 (1984): 249-285.

\bibitem{Y}
Yokoyama, K., {\em Complete classification of periodic maps on compact surfaces}, Tokyo Journal of Mathematics 15.2 (1992): 247-279.

\end{thebibliography}

\end{document}